\documentclass[11pt,reqno,a4paper]{amsart}

\usepackage[utf8]{inputenc}

\usepackage{amsthm,amsmath,amsfonts,amssymb}
\numberwithin{equation}{section}  
\usepackage{mathrsfs}
\usepackage{todonotes}

\usepackage{tikz-cd}

\usepackage{mathtools}	 
\usepackage{bm}	  
\usepackage{bbm}	 

\usepackage{tensor}	 
\usepackage[colorlinks=true]{hyperref} 
\usepackage[all]{hypcap}       
 
\newcommand{\C}{\mathbb{C}}

\newcommand{\R}{\mathbb{R}}
\newcommand{\Q}{\mathbb{Q}}
\newcommand{\Z}{\mathbb{Z}}
\newcommand{\ii}{\operatorname{i}}

\newcommand{\ch}{\operatorname{ch}}

\newcommand{\vol}{\operatorname{vol}}

\newcommand{\Rea}{\operatorname{Re}}
\newcommand{\Imm}{\operatorname{Im}}

\newcommand{\del}{\partial}
\newcommand{\delbar}{\bar{\partial}}
\newcommand{\PP}{\mathbb{P}}

\newcommand{\cA}{\mathcal{A}}
\newcommand{\cB}{\mathcal{B}}

\newcommand{\cL}{\mathcal{L}}
\newcommand{\chL}{\check{L}}

\newcommand{\cC}{\mathcal{C}}

\newcommand{\cF}{\mathcal{F}}

\newcommand{\olo}{\mathcal{O}}

\newcommand{\Hom}{\operatorname{Hom}}
\newcommand{\Ext}{\operatorname{Ext}}

\newcommand{\Bl}{\operatorname{Bl}}

\newcommand{\FS}{D\!\operatorname{FS}}
\newcommand{\Fuk}{D\!\operatorname{Fuk}}

\newcommand{\Coh}{\operatorname{Coh}}

\newcommand{\opN}{\operatorname{N}}
\newcommand{\cone}{\operatorname{Cone}}
\newcommand{\Ka}{\operatorname{Ka}}
 
\newtheorem{thm}{Theorem}[section]
 
\newtheorem{prop}[thm]{Proposition}
\newtheorem{lemma}[thm]{Lemma}
\newtheorem{cor}[thm]{Theorem}

\theoremstyle{definition}
\newtheorem{definition}[thm]{Definition}

\theoremstyle{remark}
\newtheorem{exm}[thm]{Example}
\newtheorem{rmk}[thm]{Remark}

\title[SLag smoothings, Calabi ansatz and slope inequalities]{Special Lagrangian smoothings, Calabi ansatz and stability conditions}
\author{Jacopo Stoppa}
 
\date{\today}

\begin{document}

\maketitle

\begin{abstract} As part of his work on special Lagrangian (\emph{sLag}) submanifolds with isolated conical singularities, Joyce proved a criterion for the existence of sLag smoothings, along a small variation of complex structure, for the union of two connected, compact, embedded sLags, with the same phase, intersecting transversely.

Here we construct infinitely many examples of pairs of non-compact, embedded sLags, of the same phase and with arbitrary dimension, intersecting only at infinity in a non-transverse way, which satisfy Joyce's criterion: along a small variation of complex structure, a sLag smoothing of their union exists on the stable locus where a slope inequality for periods of the holomorphic volume form holds. At least under a natural symmetry assumption, this slope inequality is also necessary for the existence of such smoothing. Our approach uses the Leung-Yau-Zaslow transform and the analysis of deformed Hermitian Yang-Mills connections with Calabi ansatz, due to Jacob and Sheu. In the unstable case, we prove that if a family of Lagrangian smoothings evolving under the natural Calabi-symmetric version of the mean curvature flow (due to Chan and Jacob) admits a limit, then this must be the union of the original sLags.

As an application we show that in our examples, in dimension two, the condition for the existence of the sLag smoothing is in fact equivalent to the stability of the corresponding object in the Fukaya-Seidel category, with respect to a known Bridgeland stability condition imported from algebraic geometry, and in the unstable case the limit of the Calabi-symmetric mean curvature flow in our result coincides with the Harder-Narasimhan decomposition, consistently with a general conjecture of Joyce. A similar (although weaker) result also holds in dimension three. 
\end{abstract}
\section{Introduction and main results}
\subsection{Special Lagrangian smoothings} We begin by recalling a fundamental result of D. Joyce concerning the existence of smoothings for the union of special Lagrangian submanifolds, in a particular case, following the exposition in \cite{Joyce_ConicSLagSurvey}, Section 9.3.

Let $(M, J, \omega, \Omega)$ denote an \emph{almost Calabi-Yau manifold} of dimension $n$, i.e. a K\"ahler manifold $(M, \omega)$ endowed with a holomorphic volume form $\Omega$. 

Suppose $L_1$, $L_2 \subset M$ are \emph{compact}, embedded \emph{special Lagrangian (sLag)} submanifolds with the same phase $e^{\ii \hat{\theta}}$, i.e. connected embedded submanifolds satisfying   
\begin{equation*}
\omega|_{L_i} = \Imm e^{-\ii \hat{\theta}}\Omega|_{L_i} = 0,
\end{equation*}
which intersect \emph{transversely} at a single point $p \in M$. 
\begin{rmk} We take all our embedded sLags to be connected.
\end{rmk}
Fix a smooth family $\{(M^s, J^s, \omega^s, \Omega^s)\!: s \in \cF\}$ of deformations of $M$, satisfying the cohomological condition $[\omega^s]|_{L_i} = 0$, $i =1, 2$, $s \in \cF$. Write
\begin{align*}
\int_{[L_i]}\Omega^s = R^s_i e^{\ii\hat{\theta}^s_i}, \, \int_{[L_1] + [L_2]} \Omega^s = R^s e^{\ii \hat{\theta}^s},\,i =1,2,
\end{align*}
where $R^s_i, R^s > 0$ and $\hat{\theta}^s_i, \hat{\theta}^s \in \R$ are chosen so that they depend continuously on $s$ with $\hat{\theta}^0_i = \hat{\theta}^0 = \hat{\theta}$ (this is possible up to making $\cF$ smaller).   
\begin{rmk}
Note that in fact, in this situation, the sLags $L_i \subset M$ deform to sLags $L^s_i \subset M^s$ for $s \in \cF$, and we have
\begin{equation*}
\int_{[L^s_i]}\Omega^s = \int_{[L_i]}\Omega^s. 
\end{equation*}
\end{rmk}
Assume that the intersection point $p$ is of \emph{type 1}, i.e. in Darboux coordinates at $p$ the tangent spaces at $L_1$, $L_2$ have canonical forms given by
\begin{align*}
\Pi^0 = \{(x_1, \ldots, x_m)\!: x_j \in \R\},\,\Pi^{\phi} = \{(e^{\ii \phi_1} x_1, \dots, e^{\ii \phi_n}x_n)\!: x_j \in \R\},
\end{align*}
where $\phi_1 \leq \phi_2 \leq \cdots \leq \phi_n$ satisfy
\begin{equation*}
\phi_1 + \cdots + \phi_n = \pi 
\end{equation*}
(this notion depends on the fixed order $\{L_1, L_2\}$; reversing the order produces an intersection point of type $n -1$).

Define the \emph{stable locus} $\cF^+ \subset \cF$ as
\begin{equation*}
\cF^+ = \{s\in \cF\!: \hat{\theta}^s_1 > \hat{\theta}^s > \hat{\theta}^s_2\}.
\end{equation*}
\begin{thm}[Joyce \cite{Joyce_ConicSLagSurvey}, Theorem 9.10]\label{JoyceThm} Let $n > 2$. There exists a smooth family of embedded special Lagrangian submanifolds $L^s \subset M^s$, with phase $e^{\ii \hat{\theta}^s}$, parametrised by $s \in \cF^+$, such that 
\begin{equation*}
L^s \to L_1 \cup L_2
\end{equation*} 
in the sense of currents as $s \to 0$.
\end{thm}
Joyce's construction actually provides more information on $L^s$: it is obtained by gluing a canonical local model, known as a Lawlor neck, of a suitable size into $L_1 \cup L_2$ at $p$, and thus its Hamiltonian isotopy class is given by the Lagrangian connected sum $ X^s_2 \# X^s_1$. 

Another crucial point is that the family of sLags $L^s$ cannot be extended over the unstable locus $\cF^- := \cF \setminus \cF^+$, namely, there is no sLag representative of $ X^s_2 \# X^s_1$ for $s \in \cF^-$ (see \cite{Thomas_MomentMirror}, Section 3).

\subsection{Outline} The purpose of the present work is to prove analogues of Joyce's smoothing Theorem \ref{JoyceThm} in some particular examples of a rather different nature, and to relate such examples to known algebro-geometric notions of stability. Our results are stated as Theorems \ref{MainThmIntro}, \ref{BStabCorSurf}, \ref{BStabCor}, \ref{FlowCor} and \ref{BundlesThmIntro} below.

In our case we work with special Lagrangians $\chL_1, \chL_2 \subset M$ which are \emph{non-compact} and \emph{intersect (non-transversely) only at infinity in $M$} (in a precise sense). However, the sLags $\chL_i \subset M$ still extend to smooth families of sLags $\chL^s_i \subset (M^s, \omega^s, \Omega^s)$ in a deformation of complex structure $M^s$ parametrised by $s \in \cF$, and suitable period integrals $\int_{[\chL^s_i]} \tilde{\Omega}^s$, involving classes $[\chL^s_i]$ in a rapid decay homology group and a modified volume form $\tilde{\Omega}^s$, are well defined. Introducing the stable locus 
\begin{equation*}
\cF^{+} = \{s\in \cF\!: \arg \int_{[\chL^s_2]} \tilde{\Omega}^s < \arg \int_{[\chL^s_1]} \tilde{\Omega}^s\},
\end{equation*}
we are still able to construct a smooth family of connected, embedded special Lagrangian submanifolds $\chL^s \subset M^s$ over $\cF^{+}$, such that 
\begin{enumerate}
\item[$(i)$] $\chL^s$ is asymptotic to $\chL^s_1 \cup \chL^s_2$ at infinity in $M^s$;
\item[$(ii)$] as $s \to 0$, we have
\begin{equation*}
\chL^s \to \chL_1 \cup \chL_2
\end{equation*}
smoothly in compact subsets of $M$;
\item[$(iii)$] under a natural further symmetry assumption, this family cannot be extended on the unstable locus  
\begin{equation*}
\cF^{-} = \{s\in \cF\!: \arg \int_{[\chL^s_2]} \tilde{\Omega}^s > \arg \int_{[\chL^s_1]} \tilde{\Omega}^s\}.
\end{equation*}
\end{enumerate}
In our examples, these results are not proved by gluing, but by an argument using mirror symmetry.
\subsection{Relation to Thomas-Yau conjectures} Theorem \ref{JoyceThm} fits naturally in the circle of ideas relating the existence of special Lagrangian submanifolds to notions of stability, namely the Thomas-Yau conjectures \cite{Thomas_MomentMirror, ThomasYau}, especially from the viewpoints of Joyce \cite{Joyce_SLags} and Li \cite{YangLi_ThomasYau}. Indeed this result was a primary motivation for these conjectures, see in particular \cite{Thomas_MomentMirror}, Section 3 and \cite{YangLi_ThomasYau}, Section 2.7. 

In particular, one expects that, when the Fukaya category $\Fuk(M^s, \omega^s)$ is well defined (i.e. if $M^s$ is suitably convex at infinity, see \cite{Joyce_SLags}, Conjecture 3.2), then for all $s \in \cF$ there should exist an exact triangle 
\begin{equation*}
L^s_2 \to \tilde{L}^s \to L^s_1 \to \chL^s_2[1], 
\end{equation*}
in $\Fuk(M^s, \omega^s)$, such that $\tilde{L}^s$ is isomorphic to the sLag $L^s$ as objects of $\Fuk(M^s, \omega^s)$ for $s \in \cF^+$, while $\tilde{L}^s$ cannot be represented by a sLag for $s \in \cF^-$. 

In our cases, thanks to the work of several authors (\cite{Zaslow_toricHMS, KuwagakiCohCon, SibillaCohCon, Shende_toricMirror, ZhouCohCon}) there is a well defined Fukaya-Seidel category $\FS(T, W^s)$ attached to $M^s$ (more precisely, as we will recall, to a corresponding Landau-Ginzburg model $W^s\!: T \to \C$, $T$ denoting an algebraic torus), the sLags $\chL^s_i$ define objects in $\FS(T, W^s)$, and we expect that there is an exact triangle
\begin{equation}\label{FukTriangle}
\chL^s_2 \to \widetilde{\chL^s} \to \chL^s_1 \to \chL^s_2[1],  
\end{equation}
in $\FS(T, W^s)$ with the same properties. However, at present, due to the actual construction of the Fukaya-Seidel category on which we rely (i.e. the coherent-constructible approach of \cite{Zaslow_toricHMS}), it is not even clear that our sLag smoothings $\chL^s$ define objects in $\FS(T, W^s)$: see the discussion in Section \ref{ExactSec}.
\subsection{Toric mirrors}\label{MirrorIntroSec} Our results concern the case when the almost Calabi-Yau manifold $M$ is obtained as the mirror of a toric Fano manifold. The idea of studying versions of the Thomas-Yau conjectures in this setting was proposed by Collins and Yau (see \cite{CollinsYau_momentmaps_preprint}, Section 9) and pursued e.g. in \cite{J_toricThomasYau, J_sLagStability} (see also \cite{FanSLags} for recent related results).

Let $X$ denote a toric Fano manifold, endowed with a torus-invariant K\"ahler form $\omega$, with K\"ahler potential $\varphi$. We write $x_i$ for the real part of the holomorphic coordinates on the complement of the toric boundary, isomorphic to $(\C^*)^n$, and denote by $[\omega] \in H^{1,1}(X)$ the K\"ahler class.

From a differential-geometric perspective, the \emph{Leung-Yau-Zaslow mirror} of $(X, \omega)$ is given by an explicit almost Calabi-Yau manifold $(M, \omega_M, \Omega_M)$, where $M \subset (\C^*)^n$ is a bounded open domain, $\omega_M$ is a K\"ahler form, and $\Omega_M$ is a holomorphic volume form, see e.g. \cite{CollinsYau_momentmaps_preprint}, Section 9. All these data depend on $\omega$ as a differential form, and $M$ is endowed with a special Lagrangian torus fibration over the momentum polytope $\Delta^o(\omega)$ of $(X, \omega)$. 

Using the natural identification $M \cong (S^1)^n \times \Delta^o$ induced by the moment map, with coordinates $(\tilde{\theta}_i,\,y_i = \del_{x_i} \varphi)$, the complex structure on $M$ is such that the coordinates $\tilde{w}_i = y_i + \ii \tilde{\theta}_i$ are holomorphic, while the K\"ahler metric is given by
\begin{equation*}
\tilde{g} = \sum_{i, j} u_{ij} \left(dy_i\otimes dy_j + d\tilde{\theta}_i \otimes d\tilde{\theta}_j\right),
\end{equation*}
where $u_{ij}$ denotes \emph{symplectic potential}, i.e. the Legendre transform of the K\"ahler potential $\varphi$. The holomorphic volume form is given by 
\begin{equation*}
\Omega_M = d\tilde{w}_1 \wedge \cdots \wedge d\tilde{w}_n.
\end{equation*}
Note that setting $\tilde{z}_i = e^{\tilde{w}_i}$ defines a biholomorphism of $M$ with a bounded open domain in $(\C^*)^n$. 

Algebro-geometrically, the mirror can be thought of as the algebraic torus $T = (\C^*)^n$, endowed with a regular function $W$, the \emph{Landau-Ginzburg (LG) potential}, and a holomorphic volume form $\Omega_{T}$, both constructed explicitly from $(X, [\omega])$, depending only on the cohomology class of $\omega$ (see \cite{Givental_toric}). If $\tilde{z}_i$ denote standard holomorphic coordinates on $T$, we have 
\begin{equation*}
\Omega_{T} = \frac{d\tilde{z}_1}{\tilde{z}_1} \wedge \cdots \wedge \frac{d\tilde{z}_n}{\tilde{z}_n}.
\end{equation*}
Fix a toric fan $\Sigma$ for $X$ and let $D_i$, $i = 1, \ldots, m$ denote the toric divisors. Then, in a fixed trivialisation of the mirror family, the LG potential is given by
\begin{equation*}
W = \sum^m_{i=1} a_i( \omega ) \tilde{z}^{v_i}, 
\end{equation*}
where $v_i$ is the primitive generator of the ray of $\Sigma$ dual to $D_i$, and the coefficients $a_i$ are uniquely determined, up to rescalings of the torus variables, by the condition that, for any integral linear relation $\sum d_i v_i = 0$, corresponding to a unique curve class $[C]$ such that $D_i . [C] = d_i$, we have
\begin{equation*}
\prod_i a^{d_i}_i = e^{-2\pi \int_C\omega}. 
\end{equation*}

An important role is played by a scaling parameter $k > 0$ for the K\"ahler form $\omega \mapsto k \omega$. The mirrors $(M_k, \omega_{M, k}, \Omega_{M, k})$ and $(T, W_k, \Omega_{T})$ relative to $k \omega$ approach a \emph{large complex structure limit} for $k \gg 1$. 

We study analogues of Theorem \ref{JoyceThm} for suitable unions of embedded special Lagrangians 
\begin{equation*}
\chL_1 \cup \chL_2 \subset (M_k, \omega_{M, k}, \Omega_{M, k}).
\end{equation*}
The resulting smoothings $\chL^s$ will be given by sLag multi-sections of a deformed fibration $M^s \to \Delta^o$. 

\subsection{Main result for $X = \Bl_p \PP^n$} Our first class of examples is constructed on the blowup $X = \Bl_p \PP^n$ of $\PP^n$ at point $p$. In this case we can require that all our data enjoy a particular type of symmetry known as \emph{Calabi ansatz}, that is, roughly speaking, that they are described by real functions of a single momentum coordinate, satisfying suitable boundary conditions: this is a standard construction in K\"ahler geometry, reviewed in Section \ref{CalabiSec}. 

Fix a K\"ahler parameter $a > 1$. Let us define subsets of the K\"ahler cone $\Ka(X)$ of $X$ by
\begin{equation*}
\cF^a_k := \{k a H - k(1-s)E\!: s\in (-\varepsilon, \varepsilon)\} \subset \Ka(X),
\end{equation*}
where $H$, $E$ denote the pullback hyperplane and exceptional divisors, $\varepsilon > 0$ is sufficiently small and $k > 0$. The subsets $\cF^a_k$ will be the base spaces of our deformations. 

For each $s \in \cF^a_k$ we fix a representative $k\omega_{a, s}$ of the corresponding K\"ahler class $k a H - k(1-s)E$, where $\omega_{a, s}$ has Calabi symmetry. Then we can construct the family of almost Calabi-Yau manifolds $(M^s_k, \omega_{M^s, k}, \Omega_{M^s, k})$ mirror to $(X, k\omega_{a, s})$ in the sense of Leung-Yau-Zaslow, as well as the family of algebro-geometric mirrors $(T, W^s_k, \Omega_T)$; both are parametrised by $\cF^a_k$. 

\begin{thm}\label{MainThmIntro} There is a dense subset of K\"ahler parameters $K \subset (1, \infty)$ such that, for $a \in K$ and for some sufficiently large, fixed $k > 0$, we have
\begin{enumerate}
\item[$(i)$] there exist smooth families of sLag sections $\chL^s_1$, $\chL^s_2 \subset M^s_k$, para\-metrised by $s \in \cF^a_k$, intersecting (non-transversely) only at infinity in $M^s_k$ (in the sense of Definition \ref{IntersectionDef}), such that $\chL^0_1$, $\chL^0_2$ have the same phase $e^{\ii \hat{\theta}}$;
\item[$(ii)$] there exists a smooth family of embedded, graded special Lagrangian submanifolds $\chL^s \subset M^s_k$, parametrised by the well defined, non-empty stable locus 
\begin{equation*}
(\cF^a_k)^+ = \big\{\arg \int_{[\chL^s_2]} e^{-W^s_k} \Omega_T < \arg \int_{[\chL^s_1]} e^{-W^s_k} \Omega_T\big\} \subset \cF^a_k,
\end{equation*}   
($[\chL^s_i]$ denoting natural integration cycles as in \cite{Fang_charges, Iritani_survey}, see Section \ref{IntegrationSec}) which are asymptotic to $\chL^s_1 \cup \chL^s_2$ at infinity in $M^s_k$, and which satisfy
\begin{equation*}
\chL^s \to \chL^0_1 \cup \chL^0_2 
\end{equation*}
smoothly as $s \to 0$, locally on $M^s_k$;
\item[$(iii)$] the special Lagrangians $\chL^s$ are obtained as multi-sections with Calabi symmetry, and with the latter assumptions, there is no extension of the family $\chL^s$ over the well defined, non-empty unstable locus   
\begin{equation*}
(\cF^a_k)^- = \big\{\arg \int_{[\chL^s_2]} e^{-W^s_k} \Omega_T > \arg \int_{[\chL^s_1]} e^{-W^s_k} \Omega_T\big\} \subset \cF^a_k;
\end{equation*}   
\item[$(iv)$] the sLag sections $\chL^s_1$, $\chL^s_2$ satisfy the further property that they correspond to well defined elements in $\FS(T, W^s)$, and that the Floer cohomology group $HF^1(\chL^s_1, \chL^s_2)$ parametrising exact triangles \eqref{FukTriangle}  is nontrivial.
\end{enumerate}  
\end{thm}
\begin{rmk}\label{PhaseAnglesRmk} The sLags $\chL^s_i$ constructed in the proof of Theorem \ref{MainThmIntro} have a specific Lagrangian phase angle $\hat{\theta} \in (0, \pi) \setminus \{\frac{\pi}{2}
\}$, depending on the K\"ahler parameter $a$. All phase angles in $(0, \pi) \setminus \{\frac{\pi}{2}
\}$ are achieved for suitable values of $a$. We expect that a variant of the proof could show the existence of sLag sections of arbitrary Lagrangian phase angle, for fixed $a$, satisfying the conclusions of Theorem \ref{MainThmIntro}.
\end{rmk}
\begin{rmk} As explained in \cite{Joyce_ConicSLagSurvey}, Section 3.3, the case $n = 2$ is special in the analysis needed for gluing methods, leading to the restriction $n \geq 3$ in Theorem \ref{JoyceThm}. This plays no role in our case. 
\end{rmk}
The proof of Theorem \ref{MainThmIntro} is completed in Section \ref{SlopeInequSec}, relying on the construction of sLag (multi-)sections provided in Sections \ref{MultiSec} and \ref{sLagsSec}. This construction is based on two main tools, the Calabi ansatz (as in the work of Jacob and Sheu \cite{JacobSheu}) and the Leung-Yau-Zaslow transform, briefly recalled in Section \ref{CalabiSec}. See e.g. \cite{Lotay_GibbonsHawkingSLag} for global results on the Thomas-Yau conjectures in dimension $n = 2$, obtained by imposing a different (Gibbons-Hawking) ansatz.

\subsection{Relation to Bridgeland stability on $X =  \Bl_p \PP^n$ for $n = 2, 3$} Our main motivation for constructing the examples given by Theorem \ref{MainThmIntro} is that of establishing a precise relation between smoothing results for a union of sLags $\chL_1 \cup \chL_2$, as in Joyce's Theorem \ref{JoyceThm}, and Bridgeland stability of Lagrangians $\chL$ fitting into an exact triangle \eqref{FukTriangle}, as predicted by Joyce's version of the Thomas-Yau conjectures (see in particular \cite{Joyce_SLags}, Conjecture 3.2). 

Working on the mirror $M$ of $X = \Bl_p \PP^n$ with Calabi ansatz greatly simplifies the construction of sLag multi-sections, and of stability conditions on $\FS(T, W^s)$ when $n =2, 3$ (compare e.g. with the global results of \cite{Lotay_GibbonsHawkingSLag} using the Gibbons-Hawking ansatz, for which stability conditions are not known). However, imposing the Calabi ansatz forces one to consider non-transverse intersections at infinity, as in Theorem \ref{MainThmIntro}. This approach allows us to solve the Bridgeland stability problem for our examples in dimension $n = 2$ (Theorem \ref{BStabCorSurf} below), and to make progress in dimension $n =3$ (Theorem \ref{BStabCor}). 
\subsubsection{$n = 2$} Fix $X = \Bl_p \PP^2$. By the general theory of Bridgeland stability conditions on projective surfaces \cite{ArcaraBertram}, for any $[\omega] \in \Ka(X)$ there exist distinguished, geometric stability conditions $(\cA, Z)$ on the bounded derived category $D^b(X)$ of coherent sheaves on $X$, where $\cA \subset D^b(X)$ denotes the heart of a suitable bounded t-structure (obtained from the standard heart $\Coh(X)$ via tilting), while the central charge is given by 
\begin{equation*}
Z(E) = -\int_X e^{-\ii \omega}\ch(E).
\end{equation*} 

Toric mirror symmetry \cite{Zaslow_toricHMS, KuwagakiCohCon, SibillaCohCon, Shende_toricMirror, ZhouCohCon} gives an equivalence
\begin{equation*}
D^b(X) \cong \FS(T, W^s), 
\end{equation*}
so there are induced stability conditions on $\FS(T, W^s)$, which we denote by $(\check{\cA}^s, \check{Z}^{s})$.
\begin{cor}\label{BStabCorSurf} Suppose we are in the situation of Theorem \ref{MainThmIntro}, in dimension $n = 2$. Then
\begin{enumerate}
\item[$(i)$] the special Lagrangian sections $\chL^s_1$, $\chL^s_2 \subset M^s_k$ give well defined objects in the heart $\cA^s \subset \FS(T, W^s)$ of the stability condition $(\check{\cA}^s, \check{Z}^{s})$ corresponding  to the K\"ahler class $[k \omega_{a, s}]$, which are $(\check{\cA}^s, \check{Z}^{s})$-stable;
\item[$(ii)$] the locus of complex structures $(\cF^a_k)^+ \subset \cF^a_k$ for which the sLag smoothing $\chL^s$ exists coincides with the locus where all nontrivial extensions of $\chL^s_1$ by $\chL^s_2$, corresponding to nonzero elements in the nontrivial group $\Ext^1_{\check{\cA}^s}(\chL^s_1, \chL^s_2)$, are $(\check{\cA}^s, \check{Z}^{s})$-stable. 
\end{enumerate}
\end{cor}
Theorem \ref{BStabCorSurf} is proved in Section \ref{SlopeInequSec}.
\subsubsection{$n = 3$} Fix $X = \Bl_p \PP^3$. Let $(\cA, Z^{0, \Gamma})$ be the distinguished geometric stability condition on $D^b(X)$ constructed in \cite{BernardaraMacri_threefolds}, Section 4.B. Here $\cA \subset D^b(X)$ denotes the heart of a suitable bounded t-structure (obtained from the standard heart $\Coh(X)$ via double tilting), while the central charge 
\begin{equation*}
Z^{0, \Gamma} \in \Hom(\operatorname{Num}_{\Q}(X), \C)
\end{equation*}
is given by 
\begin{equation*}
Z^{0, \Gamma}(E) = -\int_X e^{-\ii \omega_0}\ch(E) + \Gamma \cdot \ch_1(E),
\end{equation*}
for suitable choices of a K\"ahler class $[\omega_0]$ and of a cycle $\Gamma \in A_1(X)_{\R}$. In the construction of \cite{BernardaraMacri_threefolds}, one has
\begin{equation*}
[\omega_0] = -\frac{1}{2} K_X = 2 H - E,\,\Gamma = \frac{H^2 + E^2}{6} + C_0 \omega^2,
\end{equation*}
for a suitable choice of $C_0 \in \R$. 

By the fundamental deformation property of Bridgeland stability, there exist nearby stability conditions $(\cA^s, Z^{s, \Gamma})$ with central charges
\begin{equation*}
Z^{s, \Gamma}(E) = -\int_X e^{-\ii \omega_{a, s}}\ch(E) + \Gamma \cdot \ch_1(E) 
\end{equation*}  
for $[\omega_{a, s}] = a H - (1-s)E$ sufficiently close to $[\omega_0]$.

By toric mirror symmetry $D^b(X) \cong \FS(T, W^s)$ there are induced stability conditions on $\FS(T, W^s)$, which we denote by $(\check{\cA}^s, \check{Z}^{s, \Gamma})$.
\begin{cor}\label{BStabCor} Suppose we are in the situation of Theorem \ref{MainThmIntro}, in dimension $n = 3$. Then
\begin{enumerate}
\item[$(i)$] the special Lagrangian sections $\chL^s_1$, $\chL^s_2 \subset M^s_k$ give well defined objects in $\FS(T, W^s)$, so in particular there are corresponding elements $k^{-1} \chL^s_i \in \operatorname{Num}_{\R}(X)$;
\item[$(ii)$] suppose that the Lagrangian phase angle $\hat{\theta}$ of $\chL^0_i$ is sufficiently close to $0$ or $\frac{\pi}{2}$ from above, and that the sLag smoothing $\chL^s$ of $\chL^0_1 \cup \chL^0_2$ exists. Then the central charges 
\begin{equation*}
\check{Z}^{s, \Gamma}(k^{-1}\chL^s_1), \check{Z}^{s, \Gamma}(k^{-1}\chL^s_2),
\end{equation*}
with respect to the stability conditions $(\check{\cA}^s, \check{Z}^{s, \Gamma})$ on $\FS(T, W^s)$, lie in the upper half plane and satisfy 
\begin{equation*}
\arg Z^{s, \Gamma}(k^{-1}\chL^s_2) < \arg Z^{s, \Gamma}(k^{-1}\chL^s_1).
\end{equation*}
\end{enumerate}
\end{cor}
We regard Theorem \ref{BStabCor} as a nontrivial compatibility property with Bridgeland stability satisfied by the construction provided by Theorem \ref{MainThmIntro}. Indeed if the triangle \eqref{FukTriangle} is induced by an exact sequence in the heart $\cA'$ of \emph{any} stability condition $(\cA', Z')$, such that $\widetilde{\chL^s}$ is $(\cA', Z')$-stable, then $Z'(\chL^s_1), Z'(\chL^s_2)$ lie in the semi-closed upper half plane and we have $\arg Z'(\chL^s_2) < \arg Z'(\chL^s_1)$; our result checks these latter, weaker conditions, with respect to the known stability conditions $(\check{\cA}^s, \check{Z}^{s, \Gamma})$ (after scaling $\omega$ and $\chL^s_i$ by $k^{-1}$, which does not change the argument of central charges).  

Theorem \ref{BStabCor} is proved in Section \ref{SlopeInequSec}.
\begin{rmk} It is shown in \cite{J_sLagStability} that semistability with respect to suitable Bridgeland stability conditions implies the sLag condition for sections of mirrors of certain toric threefolds, of the form of $\chL_i$ above. However this result cannot be applied to $(\check{\cA}^s, \check{Z}^{s, \Gamma})$.     
\end{rmk}
\subsection{Momentum mean curvature flow} Building on the work of Thomas and Yau \cite{ThomasYau}, Joyce (\cite{Joyce_SLags}, Section 3.2) conjectured that the mean curvature flow with surgeries starting from a suitable class of Lagrangians, known as Lagrangian branes $\chL$, in a Calabi-Yau manifold, exists for all times and converges to a sum of special Lagrangian integral currents, which is isomorphic to the Harder-Narasimhan decomposition of $\chL$ with respect to a suitable Bridgeland stability condition. We will prove a simple result in our context, motivated by this general expectation.

Fix $X = \Bl_p \PP^n$ endowed with a Calabi-symmetric K\"ahler form $k \omega_{a, s}$ as above. As we recall in Section \ref{FlowSec}, by the work of Chan and Jacob \cite{ChanJacob} on the line bundle mean curvature flow on $X$, there exists a geometric flow in the mirror $(M^s_k, \omega_{M^s, k}, \Omega_{M^s, k})$, which gives a natural analogue of the mean curvature flow adapted to the Calabi ansatz, and is well defined on Lagrangian multi-sections (see \eqref{MCF}); here we call this the \emph{(Chan-Jacob) momentum mean curvature flow}. 
\begin{cor}\label{FlowCor} Suppose we are in the situation of Theorem \ref{MainThmIntro}, in the unstable region $s \in (\cF^a_k)^-$. Fix a Lagrangian $\chL \subset (M^s_k, \omega_{M^s, k}, \Omega_{M^s, k})$, asymptotic to $\chL^s_1 \cup \chL^s_2$ at infinity in $M^s_k$, and contained in the subset of $M^s_k$ bounded by the sections $\chL^s_1$, $\chL^s_2$. Suppose that $\chL$ is a multi-section with a single critical point (such $\chL$ always exist; morally, our assumptions mean that $\chL$ should fit in a triangle \eqref{FukTriangle} in $\FS(T, W^s)$). If the momentum mean curvature flow $\chL_t$ starting from $\chL$ exists for all times $t \in [0, \infty)$, preserves the single critical point condition, and converges smoothly, locally on $M^s_k$, then we must have
\begin{equation*}
\chL_t \to \chL^s_1 \cup \chL^s_2  
\end{equation*}  
as $t \to \infty$, smoothly, locally on $M^s_k$. In particular, in dimension $n = 2$, $\chL_t$ converges to the Harder-Narasimhan decomposition $\chL^s_1 \cup \chL^s_2$, with respect to the stability condition $(\check{\cA}^s, \check{Z}^{s})$, of any nontrivial extension of $\chL^s_1$ by $\chL^s_2$, corresponding to nonzero elements in the nontrivial group $\Ext^1_{\check{\cA}^s}(\chL^s_1, \chL^s_2)$.
\end{cor}  
Theorem \ref{FlowCor} is proved in Section \ref{FlowSec}. It is known that the momentum mean curvature flow develops finite time singularities (just as the Lagrangian mean curvature flow), see \cite{ChanJacob}, Theorem 1.1. Thus in general one should consider a version of the momentum mean curvature flow with surgeries. See \cite{LotayOliv_MCF} for global results on the Lagrangian mean curvature flow with surgeries in dimension $n = 2$ using the Gibbons-Hawking ansatz.
\subsection{Main result for split Fano bundles}
Fix $r \geq 0$, $m \geq 1$. Let 
\begin{equation*}
X_{r, m} = \PP(\olo_{\PP^m} \oplus (\olo_{\PP^m}(-1))^{\oplus r +1})
\end{equation*}
denote the projective completion of the split vector bundle $(\olo_{\PP^m}(-1))^{\oplus r +1}$ over $\PP^m$. This is Fano for $r < m$. Note that we have $X = \Bl_p \PP^n = X_{0, n-1}$.
 
Define divisor classes $[D_{\infty}]$, $[D_H]$ on $X_{r, m}$ given by $c_1(\olo_{X_{r, m}}(1))$ and by the pullback of the hyperplane class in $\PP^m$. The K\"ahler cone of $X_{r, m}$ is parametrised by
\begin{equation*}
[\omega_{\xi_1, b}] \in \xi_1 [D_H] + b [D_{\infty}],\,\xi_1,\,b > 0.
\end{equation*}

Fix a K\"ahler parameter $\xi_1 > 0$. We define subsets of the K\"ahler cone by
\begin{equation*}
\cF^{\xi_1, b}_k := \{k (\xi_1 +s) [D_H] + k b [D_{\infty}]\!: s \in ( - \varepsilon,  \varepsilon)\} \subset \Ka(X_{r, m}),
\end{equation*}
where $\varepsilon > 0$ is sufficiently small and $k > 0$. 

For each $s \in \cF^{\xi_1, b}_k$ we fix a representative $k \omega_{s, b}$ of the corresponding K\"ahler class $k (\xi_1 +s) [D_H] + k b [D_{\infty}]$, where $\omega_{s, b} := \omega_{\xi_1 + s, b}$ has Calabi symmetry. There is a family of almost Calabi-Yau manifolds $(M^s_{k}, \omega_{M^s, k}, \Omega_{M^s, k})$ mirror to $(X_{r, m}, k\omega_{s, b})$ parametrised by $\cF^b_k$, as well as the family of algebro-geometric mirrors $(T, W^s_k, \Omega_T)$.
\begin{thm}\label{BundlesThmIntro} Suppose the rank $r + 2$ of the split Fano bundle $X_{r, m}$ is even. There is a dense subset of K\"ahler parameters $K \subset (0, \infty)$ such that, for $\xi_1 \in K$, for $b > 0$ sufficiently small and for some sufficiently large $k > 0$, depending on $\xi_1$, we have 
\begin{enumerate}
\item[$(i)$] there exist smooth families of special Lagrangian sections $\chL^s_1$, $\chL^s_2 \subset M^s_k$, parametrised by $s \in \cF^{\xi_1, b}_k$, intersecting (non-transversely) only at infinity in $M^s_k$, such that $\chL^0_1$, $\chL^0_2$ have the same phase $e^{\ii \hat{\theta}}$;
\item[$(ii)$] there exists a smooth family of embedded, graded special Lagrangian submanifolds $\chL^s \subset M^s_k$, parametrised by the well defined, non-empty stable locus 
\begin{equation*}
(\cF^{\xi_1, b}_k)^+ = \{\arg \int_{[\chL^s_2]} e^{-W^s_k} \Omega_T < \arg \int_{[\chL^s_1]} e^{-W^s_k} \Omega_T\} \subset \cF^{\xi_1, b}_k,
\end{equation*}   
which are asymptotic to $\chL^s_1 \cup \chL^s_2$ at infinity in $M^s_k$, and which satisfy
\begin{equation*}
\chL^s \to \chL^0_1 \cup \chL^0_2 
\end{equation*}
smoothly as $s \to 0$, locally on $M^s_k$;
\item[$(iii)$] the special Lagrangians $\chL^s$ are obtained as multi-sections with Calabi symmetry, and with the latter assumptions, there is no extension of the family $\chL^s$ over the well defined, non-empty unstable locus   
\begin{equation*}
(\cF^{\xi_1, b}_k)^- = \{\arg \int_{[\chL^s_2]} e^{-W^s_k} \Omega^s > \arg \int_{[\chL^s_1]} e^{-W^s_k} \Omega^s\} \subset \cF^{\xi_1, b}_k.
\end{equation*}   
\end{enumerate}
\end{thm}
Theorem \ref{BundlesThmIntro} is proved in Section \ref{BundlesSec}.
\begin{rmk} $X = \Bl_p \PP^n = X_{0, n-1}$ is obtained as a particular split Fano bundle, but one can check that Theorem \ref{BundlesThmIntro} implies Theorem \ref{MainThmIntro} only for very large values of the K\"ahler parameter $a > 1$ appearing in the latter.    
\end{rmk}

\noindent{\textbf{Acknowledgements.}} I am grateful to Soheyla Feyzbakhsh, Joel Fine, Sohaib Khalid, Jason Lotay, Helge Ruddat, Saverio Secci, Nicol\`o Sibilla and Richard Thomas for conversations related to the present work.

\section{Lagrangians with Calabi symmetry}\label{CalabiSec}
Let $(M, \omega_M, \Omega_M)$ be the almost Calabi-Yau manifold obtained from a toric Fano manifold $X$, endowed with a torus-invariant K\"ahler form $\omega$, by the construction recalled above in Section \ref{MirrorIntroSec}. 

There is a fundamental correspondence between special Lagrangian sections $\cL \subset M$ of $M \to \Delta^o$, with phase $e^{\ii \hat{\theta}}$, and Hermitian holomorphic line bundles $(L, h) \to X$, such that $h$ satisfies the \emph{deformed Hermitian-Yang Mills (dHYM)} equation
\begin{equation*}
\Imm e^{-\ii\hat{\theta}} (\omega + F(h))^n = 0,
\end{equation*}
where $F(h)$ denotes the curvature of the Chern connection. This is due to Leung-Yau-Zaslow \cite{LeungYauZaslow} (see also \cite{CollinsYau_momentmaps_preprint}, Section 9). Using the notation introduced in Section \ref{MirrorIntroSec}, the correspondence is given by
\begin{equation*}
\tilde{\theta}_i(y) = u^{ij} \del_{y_j} \log(h).
\end{equation*}
We refer to this correspondence as \emph{Leung-Yau-Zaslow (LYZ) transform}.

Jacob and Sheu \cite{JacobSheu} studied the general dHYM equation 
\begin{equation*}
\Imm e^{-\ii \hat{\theta}}(\omega + \ii \alpha)^n = 0
\end{equation*}
on $X = \Bl_p \PP^n$, to be solved for a closed $(1, 1)$-form $\alpha$, using the well known Calabi ansatz
\begin{equation*}
\omega = \ii \del \delbar u(\log |z|^2),\,\alpha = \ii \del\delbar v(\log|z|^2)\textrm{ on } \C^n \setminus \{ 0 \} \cong X\setminus (H \cup E).
\end{equation*}
\begin{rmk}
This turns out to be closely related to classical examples of Harvey and Lawson, see \cite{HarveyLawsonCalibrated}, Section III.3.B and our Lemma \ref{HarLawLem}.
\end{rmk}
Suppose the cohomology class of the K\"ahler form is fixed as
\begin{equation*}
[\omega] = a H - b E,\, a > b >0.
\end{equation*}
Regarding $X\setminus (H \cup E)$ as the complement of the zero section in the total space of $\olo(-1) \to \C^{n-1}$, the moment map with respect to $\omega$ of the $U(1)$-action rotating the fibres is given by
\begin{equation*}
u'\!: X\setminus (H \cup E) \to (b, a).
\end{equation*}   
In the following we will sometimes write $\omega_{a, b}$ to emphasise the K\"ahler parameters. We let $M_{a, b} \subset (\C^*)^n$ denote the almost Calabi-Yau bounded open domain mirror to $(X, \omega) = (\PP^n, \omega_{a,b})$. It is endowed with a sLag fibration $M_{a, b} \to \Delta^o_{a, b}$ over the momentum polytope of $(X, \omega_{a, b})$.

We can express the quantity $v'(\log|z|^2)$ in the corresponding symplectic coordinates 
\begin{equation*}
x = u'(\log|z|^2)
\end{equation*}
as
\begin{equation}\label{LegendreVtransform}
f(x) = v'(\log|z|^2).
\end{equation} 
Then the boundary conditions on $f(x)$ are given by
\begin{equation}\label{BdryCond}
\lim_{x \to b^+} f(x) = q,\,\lim_{x \to a^-} f(x) = p,
\end{equation}
where the cohomology class of $\alpha$ is fixed as
\begin{equation*}
[\alpha] = p H - q E.
\end{equation*}

As explained in \cite{JacobSheu}, Section 2, the dHYM equation is equivalent to the ODE in symplectic coordinates
\begin{equation*}
\Imm e^{-\ii \hat{\theta}} \left(1 + \ii x^{-1} f(x)\right)^{n-1} (1+ \ii f'(x)) = 0,
\end{equation*}
where $f \in C^{\infty}(b, a)$ satisfies the boundary conditions \eqref{BdryCond}. This ODE is \emph{exact}, namely, equivalent to the condition
\begin{equation}\label{dHYM}
\Imm e^{-\ii \hat{\theta}} \left(x + \ii f(x)\right)^{n} = c \in \R.
\end{equation} 
Thus, solutions of the dHYM equation with $[\alpha] = p H - q E$ correspond precisely to graphical portions of the level sets of a harmonic polynomial,  
\begin{equation*}
\Imm e^{-\ii \hat{\theta}} \left(x + \ii y\right)^{n} = c  
\end{equation*} 
lying over $[b, a]$, with  
\begin{equation*}
\Imm e^{-\ii \hat{\theta}} \left(a + \ii p\right)^{n} = \Imm e^{-\ii \hat{\theta}} \left(b + \ii q\right)^{n} = c.  
\end{equation*}

If $p, q \in \Z$, then we have $[\alpha] = - c_1(L)$, $\alpha = -\ii F(h)$ for some Hermitian holomorphic line bundle $(L, h) \to X$, so composing the (inverse) Legendre transform of $f$ with the Leung-Yau-Zaslow transform of $(L, h)$ shows that a function  $f \in C^{\infty}(b, a)$ (the \emph{momentum profile}) satisfying the boundary conditions \eqref{BdryCond} corresponds to a Lagrangian section $\cL$ of  $M_{a,b} \to \Delta^o_{a, b}$.

Similarly, for $b < b' < a$, a function $f \in C^{\infty}(b', a)$ satisfying the boundary condition
\begin{equation*} 
\lim_{x \to a^-} f(x) = p
\end{equation*}
corresponds to a Lagrangian section $\cL'$ of $M_{a,b} \to \Delta^o_{a,b}$ with boundary on a special Lagrangian fibre.

We refer to $\cL$, $\cL'$ as \emph{Lagrangian sections} (with boundary) \emph{with Calabi symmetry}. If \eqref{dHYM} holds (for some $c \in \R$) these are \emph{special Lagrangian sections} (with boundary) \emph{with Calabi symmetry}.
\begin{definition}\label{IntersectionDef} Let $\cL_1, \cL_2 \subset M_{b, a}$ denote Lagrangian sections with Calabi symmetry. We say that $\cL_1, \cL_2$ intersect non-transversely at infinity in $M_{a, b}$ if the corresponding momentum profiles $f_1, f_2$ intersect at $(b, q)$ with vertical tangent.   
\end{definition}
More generally, if $p, q \in \Q$, then $f$ as above still corresponds to a sLag $\cL$ (or $\cL'$) \emph{after rescaling} 
\begin{align*}
\omega_{a, b} \mapsto k \omega_{a, b} = \omega_{k a, k b},\, [\alpha] \mapsto k [\alpha] = k p H - k q E 
\end{align*}
for some $k > 0$, since the dHYM equation is invariant under this common rescaling.  

\begin{lemma}\label{HarLawLem}
A level set $\{\Imm e^{-\ii \hat{\theta}} \left(x + \ii y\right)^{n} = c\}$ contained in $\{ x > b \}$ corresponds to a special Lagrangian multi-section $\hat{\cL} \subset M_{a, b}$ which can be decomposed as
\begin{equation*}
\hat{\cL} = \overline{\cL'_{1}} \cup \overline{\cL'_{2}} 
\end{equation*} 
where $\cL'_{i}$ are Calabi Lagrangians with boundary, with Calabi symmetry, given by momentum profiles $f_i \in C^{\infty}(b', a)$ for some $b' \in (b, a)$. 
\end{lemma}
We refer to such $\hat{\cL}$ as a \emph{sLag multi-section with Calabi symmetry}.
\begin{proof} According to Harvey and Lawson \cite{HarveyLawsonCalibrated}, Section III.3.B, the locus $\{\Imm e^{-\ii \hat{\theta}} \left(x + \ii y\right)^{n} = c\} \subset (b, a) \times \R$ is contained in a unique special Lagrangian submanifold $\widetilde{\cL} \subset \C^n \cong \R^n \times \R^n$, invariant under the diagonal action of $SO(n)$, where we regard $(b, a) \times \R$ as contained in the first complex axis plane $\times\{0\}\subset \C \times \C^{n-1}$.     

Clearly $\hat{\cL} := \overline{\cL'_{1}} \cup \overline{\cL'_{2}} \subset M_{a, b}$ is a Lagrangian submanifold except perhaps along the critical locus $\overline{\cL'_{1}} \cap \overline{\cL'_{2}}$, and  by choosing holomorphic normal coordinates for $\omega_{M_{a, b}}$ at a point $p \in \overline{\cL'_{1}} \cap \overline{\cL'_{2}}$ we find that $\hat{\cL}$ is asymptotically close, near $p$, to Harvey-Lawson's sLag $\widetilde{\cL}\subset \C^n$ nearby a critical point of the projection $\C^n \to \R^n$. Thus $\hat{\cL}$ is in fact a special Lagrangian submanifold at all points.    
\end{proof}
\begin{rmk} The Lagrangian phase angle $\hat{\theta} \in \R$ of $\cL'_i$, $\hat{\cL}$ is only determined by the level set up to integer multiples of $\pi$. Recall that in general if $\cL$ has Lagrangian phase angle $\hat{\theta} \in \R$ then $\cL[1]$ has Lagrangian phase angle $\hat{\theta} + \pi$. 
\end{rmk}

\section{Construction of distinguished sLag multi-sections $\cL^b$}\label{MultiSec}
\subsection{Construction of $\cL^b$} We consider the case when $X = \Bl_p \PP^n$ for all $n$. We endow $X$ with a smooth family of K\"ahler forms $\omega_{a, b}$, with Calabi symmetry, such that $[\omega_{a, b}] = a H - b E$. Let us denote by $M_{a, b} \subset (\C^*)^n$ its almost Calabi-Yau mirror. 

In the present Section, for a fixed K\"ahler parameter $a \in (1, \infty)$, we construct a distinguished smooth family of sLag multi-sections $\cL^b \subset M_{a, b}$ parametrised by $b \in (0, 1)$. The key insight is that $\cL^b$ is \emph{constant} in momentum coordinates.  

SLag multi-sections of Lagrangian phase angle $\hat{\theta}\in \R$, with Calabi symmetry with respect to $\omega_{a, b}$, correspond, in the momentum coordinate $x$, to the connected components $\cC$ of the level sets 
\begin{equation*}
\{\Imm e^{-\ii \hat{\theta}} (x + \ii y)^n = c\} \subset (b, a) \times \R. 
\end{equation*}

This holds at least if the boundary conditions satisfy a suitable rationality assumption; we will check that this can be achieved in our construction in Section \ref{ScalingSec} below. We will restrict to phase angles $\hat{\theta} \in (-\pi, \pi)$. 
\begin{exm} For $n =3$, $\cC$ is a connected component of the real cubic
\begin{equation*}
\{\-x\sin (\hat{\theta}) \left(x^2-3 y^2\right)- y \cos (\hat{\theta}) \left(y^2-3 x^2\right) = c\} \subset \R^2.
\end{equation*}
\begin{figure}[h]
  \includegraphics[width=5cm]{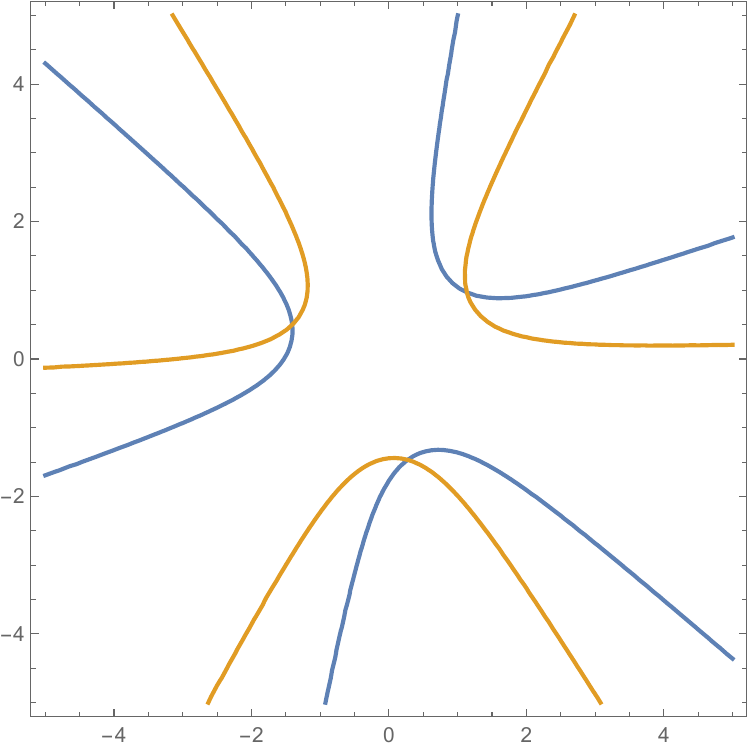}
  \caption{Some level sets $\cC$.}
  \label{fig:LevelSets}
\end{figure}
Figure \ref{fig:LevelSets} shows the cubic for $c = 3$ and $\hat{\theta} = 1$, $\hat{\theta} = 0.1$.
\end{exm}
\begin{lemma}\label{ConstructionLem1} There are explicit choices of $c \in \R$ such that the critical value of the corresponding sLag sections $\cL$ with Calabi symmetry, with phase angle $\hat{\theta} \in (-\pi, \pi)$, in momentum coordinates, is given by $x = 1$.  
\end{lemma}
\begin{proof}
The critical values $x$ and critical points $y$ are solutions of
\begin{equation*}
\del_y \Imm e^{-\ii \hat{\theta}} (x + \ii y)^n = 0,\,\Imm e^{-\ii \hat{\theta}} (x + \ii y)^n = c.
\end{equation*}

Solving for $y$, $c$ we find
\begin{equation*}
\Imm e^{-\ii \hat{\theta}} \ii n (1 + \ii y)^{n-1} = 0,\,\Imm e^{-\ii \hat{\theta}} (1 + \ii y)^n = c.
\end{equation*}
Writing $1 + \ii y = \rho e^{\ii \psi}$, we have
\begin{equation*}
\rho = (1 + y^2)^{\frac{1}{2}},\,\cos(\psi) = (1 + y^2)^{-\frac{1}{2}},\,\sin(\psi) = y(1 + y^2)^{-\frac{1}{2}} 
\end{equation*}
and we need to solve
\begin{equation*}
\Imm e^{\ii ((n-1)\psi + \frac{\pi}{2} - \hat{\theta})} = 0,\,\Imm \rho^n e^{\ii ( n \psi - \hat{\theta})} = c 
\end{equation*}
for $\psi$, $c$. This gives
\begin{align*}
& \psi = \frac{\hat{\theta} - \frac{\pi}{2} + m \pi}{n-1},\,m = 1 - n,\, \ldots,\,n-2,
\end{align*}
as well as
\begin{align*}
& c = c_m = \rho^n \sin\left(n \psi - \hat{\theta}\right) = (\cos(\psi))^{-n} \sin\left(n \psi - \hat{\theta}\right)\\
&= \frac{\sin\left(\frac{n}{n-1}(\hat{\theta} - \frac{\pi}{2} + m \pi) -\hat{\theta}\right)}{\left(\cos\left( \frac{\hat{\theta} - \frac{\pi}{2} + m \pi}{n-1} \right)\right)^{n}} = \frac{\sin\left(\frac{\hat{\theta} - \frac{\pi}{2} + m \pi}{n-1}-\frac{\pi}{2} + m\pi\right)}{\left(\cos\left( \frac{\hat{\theta} - \frac{\pi}{2} + m \pi}{n-1} \right)\right)^{n}}\\
%\end{align*}
%\begin{align*}
& = \frac{(-1)^m \sin\left(\frac{\hat{\theta} - \frac{\pi}{2} + m \pi}{n-1}-\frac{\pi}{2}\right)}{\left(\cos\left( \frac{\hat{\theta} - \frac{\pi}{2} + m \pi}{n-1} \right)\right)^{n}} = (-1)^{m+1} \left(\cos\left( \frac{\hat{\theta} - \frac{\pi}{2} + m \pi}{n-1} \right)\right)^{1- n},\\
&m = 1 - n,\, \ldots,\,n-2. 
\end{align*}
These solutions are not all distinct, but we will fix the values of $m$ appropriately in our construction. 

The corresponding critical point is given by
\begin{equation*}
q = q_m = \rho \sin(\psi) = \tan(\psi) = \tan\left( \frac{\hat{\theta} - \frac{\pi}{2} + m \pi}{n-1} \right).
\end{equation*}
\end{proof}
\begin{exm} When $n = 3$, we can impose equivalently that the discriminant $\Delta$ of the cubic $\cC$ with respect to $y$ vanishes. We compute
\begin{equation*}
\Delta = -\frac{27}{2} \left(c^2 \cos (2 \hat{\theta})+c^2-8 c \sin (\hat{\theta})-8\right).
\end{equation*}
Solving $\Delta = 0$ with respect to $c$ for fixed Lagrangian phase $\hat{\theta}$ gives
\begin{equation*}
 c = c_{\pm} = \frac{2}{\pm 1 -\sin (\hat{\theta}) }.  
\end{equation*}
We denote a corresponding connected component by $\cC_{\pm}$. The intersection points $\cC_{\pm} \cap \{1\} \times \R$ are solutions of 
\begin{equation*}
\left(3 y^2-1\right) \sin (\hat{\theta})-y \left(y^2-3\right) \cos(\hat{\theta}) = c_{\pm},
\end{equation*}
and the solution $q$ of multiplicity $2$ is given explicitly by
\begin{align}\label{qpm}
q_{\pm}(\hat{\theta}) = \frac{\sin \left(\frac{\hat{\theta}}{2}\right)\pm\cos \left(\frac{\hat{\theta} }{2}\right)}{\cos
   \left(\frac{\hat{\theta} }{2}\right)\mp\sin \left(\frac{\hat{\theta} }{2}\right)}.
\end{align}
Figure \ref{fig:TangentLine} shows the tangent line to $\cC_+$ for $\hat{\theta} = 3$.
\begin{figure}[h]
  \includegraphics[width=5cm]{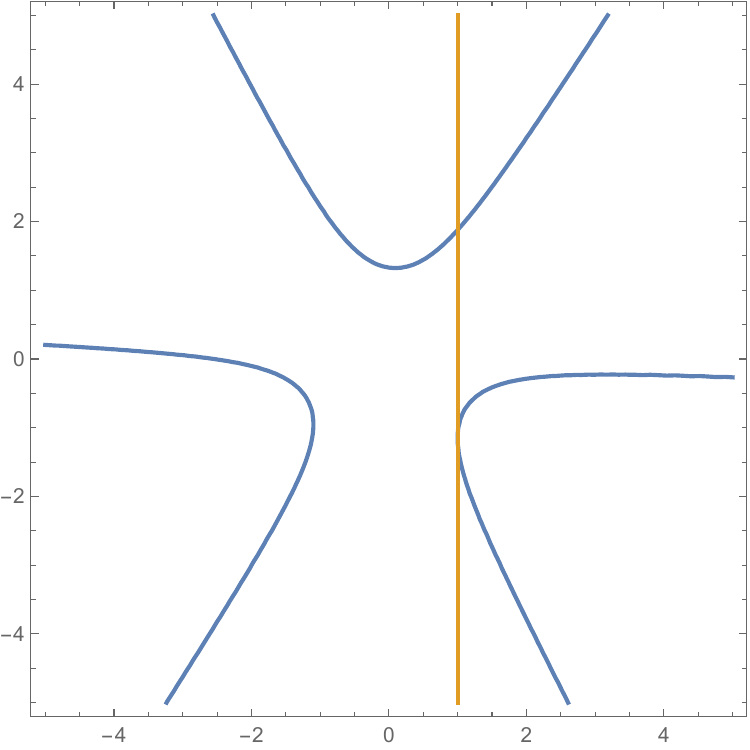}
  \caption{Splitting $\cC_+$ into graphical branches.}
  \label{fig:TangentLine}
\end{figure}
\end{exm}
\begin{lemma}\label{ConstructionLem2} Fix $c \in \R$ as in Lemma \ref{ConstructionLem1}. There are explicit choices of the K\"ahler parameter $a > 0$ and Lagrangian phase angle $\hat{\theta} \in (-\pi, \pi)$ for which the polynomial 
\begin{equation*}
\Imm e^{-\ii \hat{\theta}} (a + \ii y)^n - c 
\end{equation*}
has vanishing constant term. 
\end{lemma}

This step simplifies radically several computations in our proof of Theorem \ref{MainThmIntro}, and, as we will see shortly, it has a natural cohomological interpretation.
\begin{proof}
Using the notation introduced in the proof of Lemma \ref{ConstructionLem1}, we impose
\begin{equation*}
- \sin(\hat{\theta}) a^n_m = c_m, 
\end{equation*}
giving
\begin{equation*}
a_m = \left(-\frac{c_m}{\sin(\hat{\theta})}\right)^{\frac{1}{n}} = \left((-1)^{m}\sin(\hat{\theta})\left(\cos\left( \frac{\hat{\theta} - \frac{\pi}{2} + m \pi}{n-1} \right)\right)^{n-1}\right)^{-\frac{1}{n}}.
\end{equation*}

For each $\hat{\theta}$, we fix a value of $m$ such that the necessary and sufficient condition 
\begin{equation*}
a_m > 1
\end{equation*}
on the K\"ahler parameter $a$ is satisfied (i.e. $\omega_{a, 1}$ is K\"ahler). This choice is locally constant away from some singular values. 

In particular we can single out the ``principal branch" $m = 0$, for which
\begin{equation*}
a = \left(\sin(\hat{\theta})\left(\cos\left( \frac{\hat{\theta} - \frac{\pi}{2}}{n-1} \right)\right)^{n-1}\right)^{-\frac{1}{n}}.
\end{equation*}

An elementary study then shows that for $n \geq 2$ we have
\begin{equation*}
\hat{\theta} \in (0, \pi) \setminus \{\frac{\pi}{2}\}\,\Rightarrow\, 0 < \sin(\hat{\theta})\left(\cos\left( \frac{\hat{\theta} - \frac{\pi}{2}}{n-1} \right)\right)^{n-1} < 1  
\end{equation*}
(see Figure \ref{fig:nDimKahlerParam} for the case $n = 4$).
\begin{figure}[h]
  \includegraphics[width=5cm]{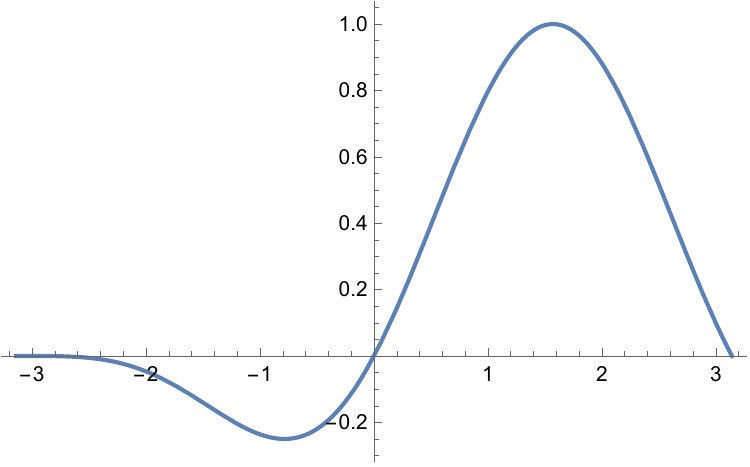}
  \caption{}
  \label{fig:nDimKahlerParam}
\end{figure} 
\end{proof}
\begin{definition} Fix $c$ as in Lemma \ref{ConstructionLem1} and $a, \, \hat{\theta}$ as in Lemma \ref{ConstructionLem2}. We define $\cL^b \subset M_{a, b}$ as the family of sLag multi-sections, parametrised by $b \in (0, 1)$, corresponding to the connected component $\cC$ of the level set $\Imm e^{-\ii \hat{\theta}}(x + \ii y) = c$ containing the point $(a, 0)$. The Lagrangian phase angle of $\cL^b$ is $\hat{\theta} \in (0, \pi)$.
\end{definition}  

\subsection{Boundary conditions} We need to understand the boundary conditions satisfied by $\cL^b$. These are determined by the solution $y = y(a, c, \hat{\theta})$ to
\begin{equation*}
\Imm e^{-\ii \hat{\theta}} (a + \ii y)^n = c,\, (a, y) \in \cC. 
\end{equation*}
Recalling that by construction we have $\Imm e^{-\ii \hat{\theta}} a^n = c$, we write this in the form
\begin{equation*}
\Imm e^{-\ii \hat{\theta}} a^n\left(1 + \ii p\right)^n = c,\,p := \frac{y}{a}.
\end{equation*}
In general, given $z \in \C^*$, the condition 
\begin{equation*}
\Imm (z\lambda) = \Imm z
\end{equation*}
is equivalent to
\begin{equation*}
\frac{\bar{\lambda} - 1}{\lambda - 1} = \frac{z}{\bar{z}},
\end{equation*}
and so 
\begin{equation*}
\lambda = 1 - r e^{- \ii \arg(z)},\,r > 0.
\end{equation*}
So our auxiliary parameter $p$, which determines the boundary conditions satisfied by $\cL^b$, solves
\begin{equation*}
\left(1 + \ii p \right)^n = 1 - r e^{-\ii \arg\left(e^{-\ii \hat{\theta}} a \right)} = 1 - r e^{\ii \hat{\theta}},
\end{equation*}
or equivalently
\begin{equation*}
\hat{\theta} = \arg \left(1 - \left(1 + \ii p \right)^n\right) \in (-\pi, \pi).
\end{equation*}
We note that by an elementary computation the interval 
\begin{equation*}
p \in \left(-\tan\left(\frac{\pi}{n}\right), 0\right)
\end{equation*}
corresponds to 
\begin{equation*}
\hat{\theta} \in \left(0, \frac{\pi}{2}\right),\,\hat{\theta} = \arctan\left(\frac{ \Imm \left(1 - \left(1 + \ii p \right)^n\right)}{\Rea \left(1 - \left(1 + \ii p\right)^n\right)}\right) 
\end{equation*} 
while
\begin{equation*}
p \in \left(0, \tan\left(\frac{\pi}{n}\right)\right)
\end{equation*}
corresponds to 
\begin{equation*}
\hat{\theta} \in \left(\frac{\pi}{2}, \pi \right),\,\hat{\theta} = \arctan\left(\frac{ \Imm \left(1 - \left(1 + \ii p \right)^n\right)}{\Rea \left(1 - \left(1 + \ii p\right)^n\right)}\right) + \pi.
\end{equation*} 

We claim that for $p \in \left(-\tan\left(\frac{\pi}{n}\right), \tan\left(\frac{\pi}{n}\right)\right)$ we have
\begin{equation*}
a + \ii a p \in \cC,
\end{equation*}
that is, the points $a$, $a + \ii a p$ lie in the same connected component of the locus $\Imm e^{-\ii \hat{\theta}}(x + \ii y)^n = c$. 

Indeed, as explained by Jacob and Sheu \cite{JacobSheu}, the connected components of 
\begin{equation*}
\{\Imm e^{-\ii \hat{\theta}}(x + \ii y)^n = c\} \subset \R^2
\end{equation*}
are analytic curves asymptotic to the union of rays
\begin{equation*}
\{\Imm e^{-\ii \hat{\theta}}(x + \ii y)^n = 0\} \subset \R^2,
\end{equation*}
which moreover lie in \emph{alternating} wedges. See Figure \ref{fig:Wedges} for the case $n = 4$, $\hat{\theta} = 1$, $c = 1$.
\begin{figure}[h]
  \includegraphics[width=5cm]{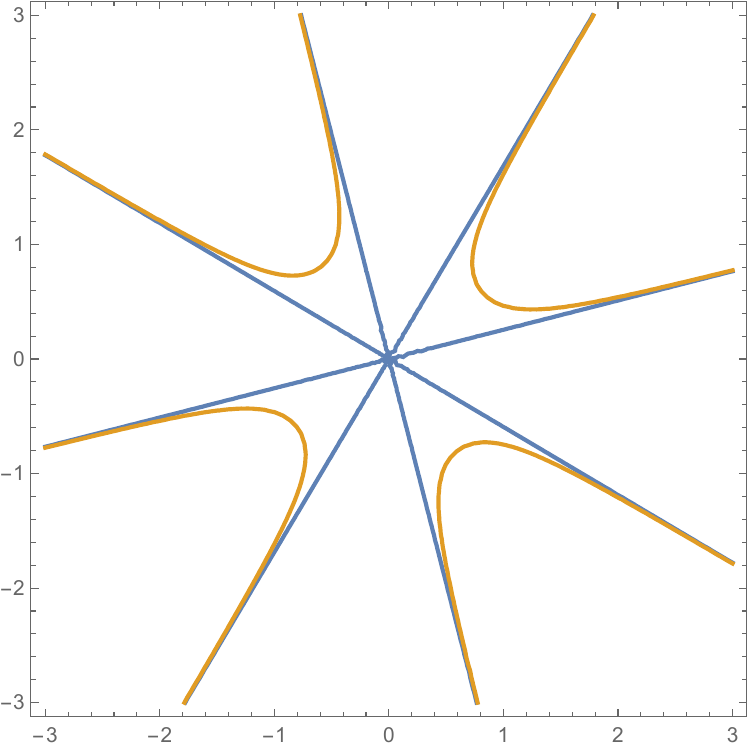}
  \caption{Structure of the level set for $n = 4$.}
  \label{fig:Wedges}
\end{figure} 
It follows that, if we have 
\begin{equation*}
|\arg(a + \ii a p)| < \frac{\pi}{n},  
\end{equation*}
then $a + \ii a p$ necessarily lies in the connected component $\cC$. Since $a$ is real and positive, the condition on the argument is equivalent to 
\begin{equation*}
|\arg(1 + \ii p)| < \frac{\pi}{n}, 
\end{equation*}
which is satisfied tautologically for $p \in \left(-\tan\left(\frac{\pi}{n}\right), \tan\left(\frac{\pi}{n}\right)\right)$.

The upshot of our computations is that the boundary conditions satisfied by $\cL^b$ are given in momentum coordinates by
\begin{equation*}
\cC \cap \{a\} \times \R = \{a,\, a (1 + \ii p)\},\,p \in \left(-\tan\left(\frac{\pi}{n}\right), \tan\left(\frac{\pi}{n}\right)\right).
\end{equation*}
\begin{exm} When $n = 3$, by our construction, $y = a p$ is a nonzero solution of the cubic equation $\Imm e^{-\ii \hat{\theta}}(a + \ii y)^3 = c$, where $c$ is chosen precisely so that the cubic has vanishing constant term. So $p, \hat{\theta}$ solve 
\begin{equation*} 
p^2 = 3 \tan (\hat{\theta})p + 3.
\end{equation*} 
For $\hat{\theta} \in (0, \frac{\pi}{2})$, the unique solution $p \in \left(-\tan\left(\frac{\pi}{3}\right), \tan\left(\frac{\pi}{3}\right)\right)$ is given by  
\begin{equation*}
p = \frac{1}{2} \left(3 \tan (\hat{\theta}) - \left(9 \tan ^2(\hat{\theta})+12\right)^{\frac{1}{2}}\right).
\end{equation*}
Figure \ref{fig:DirectSum1} shows $\cC$ together with the lines $x = 1$, $x = a$ for $\hat{\theta} = 0.1$.
\begin{figure}[h]
  \includegraphics[width=5cm]{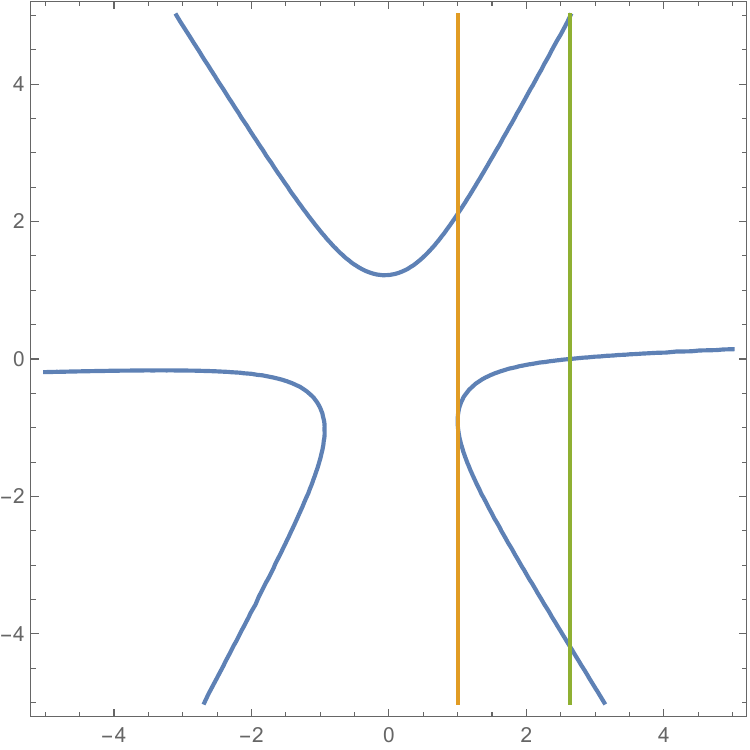}
  \caption{Imposing boundary conditions.}
  \label{fig:DirectSum1}
\end{figure}
\end{exm}
\begin{exm}\label{SurfExm} Similarly, in the simplest case $n = 2$, $y = a p$ solves the quadratic equation
\begin{equation*}
\Imm e^{-\ii \hat{\theta}}(a + \ii y)^2 = c
\end{equation*}
for $c$ killing the constant term, from which we compute directly
\begin{equation*}
p = -2 \cot(\hat{\theta}),
\end{equation*} 
which, together with the relations
\begin{align*}
& a = \left(\sin(\hat{\theta}) \cos\left( \hat{\theta} - \frac{\pi}{2}\right) \right)^{-1} = \frac{1}{\sin^2
(\hat{\theta})},\\
& q = \tan\left( \hat{\theta} - \frac{\pi}{2} \right),
\end{align*}
determines all our parameters in terms of the angle $\hat{\theta}$.  
\end{exm}
\subsection{Scaling action and rationality}\label{ScalingSec} There is a natural scaling action on our whole construction, which simply replaces $\omega_{a, b}$ with $\omega_{k a, k b}$ for all $b \in (0, 1)$ for any $k > 0$. This leaves the Lagrangian phase angle $\hat{\theta}$ invariant. The critical point of our multi-section, now lying over the critical momentum value $x = k$, is given by
\begin{equation*}
k q = k \tan\left( \frac{\hat{\theta}(p) - \frac{\pi}{2}}{n-1} \right),\,\hat{\theta} = \arg \left(1 - \left(1 + \ii p \right)^n\right) \in (0, \pi) \setminus \{\frac{\pi}{2}\}.
\end{equation*} 
Similarly the boundary values over the momentum K\"ahler parameter
\begin{equation*}
k a = k \left(\sin(\hat{\theta})\left(\cos\left( \frac{\hat{\theta}(p) - \frac{\pi}{2}}{n-1} \right)\right)^{n-1}\right)^{-\frac{1}{n}} 
\end{equation*}
are given by
\begin{equation*}
\{0, k a p\},\,\,p \in \left(-\tan\left(\frac{\pi}{n}\right), \tan\left(\frac{\pi}{n}\right)\right).
\end{equation*}

Thus, the sLag multi-section $\cL^b$ is well defined provided that the critical point and the boundary conditions satisfy 
\begin{equation*}
k q,\,k a p \in \Z.
\end{equation*}
We call such values of $k > 0$ \emph{admissible}. Clearly the latter condition can be achieved by scaling with a suitable $k > 0$ iff 
\begin{equation*}
\frac{a(p) p}{q(p)} \in \Q.
\end{equation*}
The function $\frac{a(p) p}{q(p)}$ is continuous and satisfies 
\begin{equation*}
\lim_{p \to \mp \tan\left(\frac{\pi}{n}\right)^{\pm}} \frac{a(p) p}{q(p)} = +\infty,\,\lim_{p \to 0} \frac{a(p) p}{q(p)} = 0,
\end{equation*}  
so $\frac{a(p) p}{q(p)} \in \Q$ holds on a dense subset of $\left(-\tan\left(\frac{\pi}{n}\right), \tan\left(\frac{\pi}{n}\right)\right)$. Moreover, we have
\begin{equation*}
\lim_{p \to \mp \tan\left(\frac{\pi}{n}\right)^{\pm}} a(p) = +\infty,\,\lim_{p \to 0} a(p) = 1,
\end{equation*}  
so the required rationality condition holds for a dense subset of K\"ahler parameters $a \in (1, \infty)$.
\section{Construction of sLag sections $\cL^b_i$}\label{sLagsSec}
In the previous Section we constructed a family of sLag multi-sections $\cL^b \subset M_{a, b}$ parametrised by $b \in (0, 1)$ for fixed $a > 1$. This family is constant in momentum coordinates and corresponds to a distinguished connected component $\cC \subset (b, a) \times \R$ of the locus $\Imm e^{-\ii \hat{\theta}}(x + \ii y)^n = c$. 

On compact subsets of $M_{a, b}$, the smooth limit of the family $\cL^b$ as $b \to 1$ is given by the \emph{disconnected} sLag corresponding to the union of level sets 
\begin{equation*}
\cC \cap (1, a) \times \R = \cC_1 \cup \cC_2, 
\end{equation*}
where we write $\cC_1$ for the connected component containing $a(1 + \ii p)$. This is the disjoint union of sLag sections $\cL_1 \cup \cL_2$ with $\cL_i \subset M_{a, 1}$. 
\begin{rmk} In momentum coordinates, $\cL_1$, $\cL_2$ correspond to sections over the momentum interval $[1, a)$ which intersect at $(1, q)$ with vertical tangent. Thus $\cL_1$, $\cL_2$ intersect non-transversely at infinity in the sense of Definition \ref{IntersectionDef}.
\end{rmk}
We consider the holomorphic line bundles on $X$ given by
\begin{equation*}
L_1 = \olo(- k a p H + k q E),\, L_2 = \olo( k q E)
\end{equation*}
where $k > 0$ is such that $k a p, k q \in \Z$. Recall this $k$ exists for $p$ lying in a dense subset of $\left(-\tan\left(\frac{\pi}{n}\right), \tan\left(\frac{\pi}{n}\right)\right)$, corresponding to a dense subset of K\"ahler parameters $a \in (1, \infty)$.

The line bundles $L_i$, $i =1, 2$ are chosen precisely so that the boundary conditions for the dHYM equation on the \emph{duals} $L^{\vee}_i$ agree with those for the sLag $\cL_i$ constructed above for $b$ close to $1$.  
\begin{lemma} The sLag sections $\cL_i \subset M_{a, 1}$ extend to families of sLag sections parametrised by all $b$ sufficiently close to $1$, satisfying the boundary conditions given in momentum coordinates by
\begin{equation*}
\cC_1 \cap \{b\} \times \R = b + \ii q,\,\cC_1 \cap \{a\} \times \R = a + \ii a p,\,
\end{equation*}  
respectively
\begin{equation*}
\cC_2 \cap \{b\} \times \R = b + \ii q,\,\cC_2 \cap \{a\} \times \R = 0.
\end{equation*}
\end{lemma}
This follows from a result of Jacob and Sheu.
\begin{thm}[\cite{JacobSheu}, Theorem 1] On $X = \Bl_p \PP^n$, the existence of a solution of the general dHYM equation
\begin{equation*}
\Imm e^{-\ii \hat{\theta}}(\omega + \ii \alpha)^n = 0,
\end{equation*}
with Calabi symmetry, is equivalent to the following numerical condition:  for all irreducible analytic subvarieties $V \subset X$, we have
\begin{equation*}
\varepsilon_V \Imm \frac{\int_V e^{-\ii \omega + \alpha}}{\int_X e^{-\ii \omega + \alpha}} > 0, 
\end{equation*}
where $\varepsilon_V = \pm 1$ depends only on $\dim V$, $[\omega]$, $[\alpha]$, in a locally constant way. 
\end{thm}
Thus, the existence of dHYM solutions with Calabi symmetry is an open condition with respect to varying $[\omega]$ for fixed $[\alpha]$. 

Applying this in our case to the line bundles $L_i$ we find that, for all $b$ sufficiently close to $1$, the dHYM equation with respect to $\omega_{k a, k b}$ is solvable on $L^{\vee}_i$ and $L_i$. Thus the sLag sections $\cL_i$ extend to sLag sections $\cL^b_i$ for $b$ sufficiently close to $1$, as claimed.

Figures \ref{fig:Defo1}, \ref{fig:Defo2} show $\cC_i$ for $n = 3$, $b = 1.1$, $b = 0.9$  respectively, for a suitable K\"ahler parameter $a \approx 1.2$, with the constant background of $\cC$; in each Figure, $\cC$ and $\cC_i$ are the rightmost branches, and the relevant momentum interval is given by the projection to the positive real axis of the intersection points $\cC_1 \cap \cC_2$, $\cC_i \cap \cC$. 
\begin{figure}
\centering
\begin{minipage}{.5\textwidth}
  \centering
  \includegraphics[width=5cm]{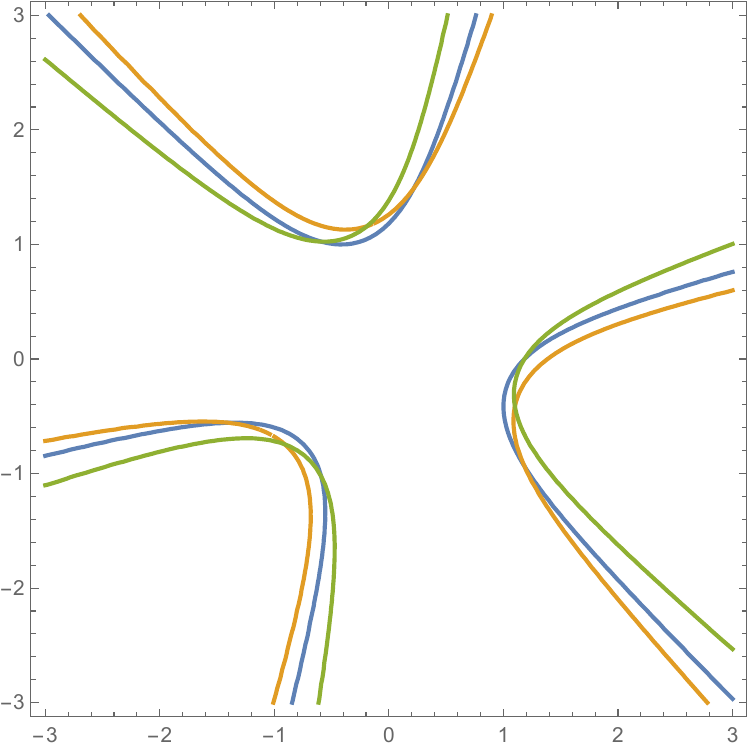}
  \caption{}
  \label{fig:Defo1}
\end{minipage}%
\begin{minipage}{.5\textwidth}
  \centering
  \includegraphics[width=5cm]{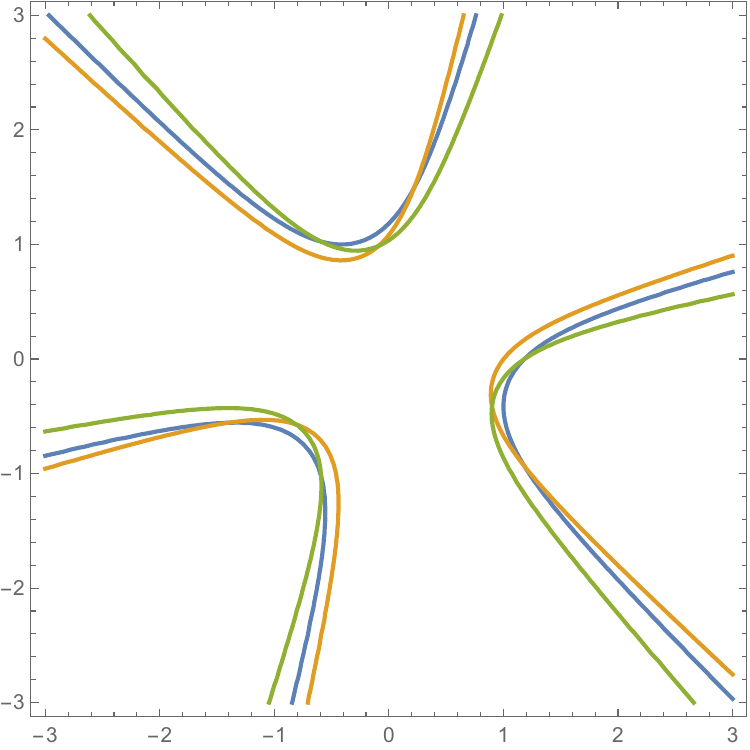}
  \caption{}
  \label{fig:Defo2}
\end{minipage}
\end{figure}

The scaling action described in Section \ref{ScalingSec} also applies to these sections, yielding sLags $\cL_{i, k} \subset M_{k a, k b}$; we suppress $k$ from the notation since this should cause no confusion.

\subsection{Lagrangian phase angle} We compute the Lagrangian phase angle of the sLag sections $\cL_i$, i.e. of $\cL^b_i$ for $b = 0$. 
\begin{lemma}\label{PhasesLem} Suppose $p \in \left(-\tan\left(\frac{\pi}{n}\right), 0\right)$, corresponding to $\hat{\theta} \in (0, \frac{\pi}{2})$. Then the Lagrangian phase angles of the sLag sections $\cL_1$, $\cL_2$ are  $-\hat{\theta} + \pi$, $-\hat{\theta}$ respectively. Similarly, for $p \in \left(0, \tan\left(\frac{\pi}{n}\right)\right)$, the Lagrangian phase angles of $\cL_1$, $\cL_2$ are  $-\hat{\theta}$, $-\hat{\theta} + \pi$.
\end{lemma}
\begin{proof}
According to Jacob and Sheu \cite{JacobSheu}, the unique lifted angle $\vartheta_i$ for the dHYM equation on $L^{\vee}_i$ can be computed in terms of the momentum profile $y_i = y_i(x)$ by
\begin{equation*}
\vartheta_i = (n-1) \arctan \frac{y_i(x)}{x} + \arctan \frac{d y_i}{dx}.
\end{equation*} 
Suppose $p \in \left(-\tan\left(\frac{\pi}{n}\right), 0\right)$. Then, in our construction, the point $a + \ii a p \in \cC_1$ lies \emph{below} the point $a \in \cC_2$, and so, moving along the branch $\cC_1$ towards the critical point $b + \ii q$ over $x = b$, we have  
\begin{equation*}
\lim_{b \to 1} \left(\lim_{x \to b^+} \frac{d y_1}{dx}\right) = - \infty,\,\lim_{b \to 1}\left(\lim_{x \to b^+} \frac{d y_2}{dx}\right) = + \infty.
\end{equation*}
So we can compute
\begin{align*}
& \vartheta_1|_{b = 1} = \lim_{x \to b^+} \left((n-1) \arctan y_i(x) + \arctan \frac{d y_i}{dx}\right)\\
& = (n-1) \arctan q + \arctan(- \infty)\\
%\end{align*}
%\begin{align*}
& = (n-1)\arctan \tan\left( \frac{\hat{\theta}(p) - \frac{\pi}{2}}{n-1} \right) - \frac{\pi}{2}\\
& = \hat{\theta} - \pi,  
\end{align*}
using $\frac{\hat{\theta}(p) - \frac{\pi}{2}}{n-1} \in (-\frac{\pi}{2}, \frac{\pi}{2})$ for $\hat{\theta} \in (0, \pi)$. Similarly,
\begin{align*}
\vartheta_2|_{b=1} = (n-1)\arctan \tan\left( \frac{\hat{\theta}(p) - \frac{\pi}{2}}{n-1} \right) + \frac{\pi}{2} = \hat{\theta}.
\end{align*} 

It follows that the LYZ mirrors of $L_1$, $L_2$, the sLag sections $\cL_1$, $\cL_2$, have Lagrangian phase angles $-\hat{\theta} + \pi$, $-\hat{\theta}$ respectively. A similar argument applies to the case $p \in \left(0, \tan\left(\frac{\pi}{n}\right)\right)$.
\end{proof}
\subsection{Properties of the family $\cL^b$} We can now summarise the main properties of the family of sLag multi-sections $\cL^b$ constructed in Section \ref{MultiSec}.
\begin{lemma}\label{PropertiesLem} The family of sLags $\cL^b$ constructed in Section \ref{MultiSec} satisfies the following properties:
\begin{enumerate}
\item[$(i)$] it is defined for all $b < 1$ sufficiently close to $1$, and each $\cL^b$ has Calabi symmetry;
\item[$(ii)$] its phase $e^{-\ii \hat{\theta}}$ equals the phases of $\cL^b_1$, $\cL^b_2[1]$, or $\cL^b_1[1]$, $\cL^b_2$ (according to the sign of $p$), at $b = 0$;
\item[$(iii)$] for fixed $b$, it is asymptotic to $\cL^b_1 \cup \cL^b_2$ at infinity in $M_{a, b}$;
\item[$(iv)$] as $b \to 1^{-}$, we have $\cL^b \to \cL_1 \cup \cL_2$ smoothly on compact subsets of $M_{a, b}$;
\item[$(v)$] it cannot be extended, with the same properties, to any value $b > 1$.
\end{enumerate}
\end{lemma}
\begin{proof} The properties $(i), (iv)$ follow directly from the construction in Section \ref{MultiSec}. For fixed $b <1$, the sLags $\cL^b \subset M_{a, b}, \cL^b_1 \cup \cL^b_2 \subset M_{a, b}$ are obtained as the composition of the Leung-Yau-Zaslow and Legendre transforms of functions on $C^{\infty}(b, a)$ which satisfy the same boundary conditions over $a$, corresponding to the non-compact end of $\cL^b$. This implies property $(iii)$. Property $(ii)$ follows from Lemma \ref{PhasesLem} and the action of the shift functor on phase angles. Finally, property $(v)$ follows from the fact the the only possible sLag multi-section with Calabi symmetry asymptotic to $\cL^b_1 \cup \cL^b_2$ would correspond, under LYZ and Legendre transforms, to the distinguished connected component $\cC$ of the locus $\Imm e^{-\ii \hat{\theta}}(x + \ii y)^n = c$ constructed in Section \ref{MultiSec}. However for $b > 1$ the connected component $\cC$ is \emph{not} contained in the locus $(b, a) \times \R$, and so does not yield a Lagrangian multi-section in $M_{a, b}$.  
\end{proof}
Figures \ref{fig:Defo1}, \ref{fig:Defo2} above illustrate the transition between the regions where $\cL^b$ exists (Figure \ref{fig:Defo2}, $b < 1$), respectively does not exist (Figure \ref{fig:Defo2}, $b > 1$).
\subsection{Integration}\label{IntegrationSec} Let $W_{a, b}$ denote the Landau-Ginzburg potential corresponding to the K\"ahler class $\omega_{a, b}$, as recalled in Section \ref{MirrorIntroSec}. Toric homological mirror symmetry provides an equivalence
\begin{equation*}
D^b(X) \cong \FS(T, W_{a, b}).  
\end{equation*}
For all $b$ sufficiently close to $1$, we have the sLag sections $\cL^b_i$, with LYZ transforms given by fixed line bundles $L_i$. Thus, through the equivalence above, we can associate to the sLag sections $\cL^b_i \subset M_{a, b}$ well defined objects in the Fukaya-Seidel category $\FS(T, W_{a, b})$, which we still denote by $\cL^b_i$, abusing notation.

We introduce the holomorphic volume form on $T$ given by 
\begin{equation*}
\Omega^b_{T, k} = e^{-W_{ka, kb}} \Omega_{T}. 
\end{equation*}
By the general results of Iritani (discussed e.g. in \cite{Iritani_survey}) and Fang (see \cite{Fang_charges}) there is a homomorphism
\begin{equation*}
\Gamma\!: K^0(X) \to H_n(T, \{\Rea W(k \omega_{a,b}) \gg 0\}; \Z)
\end{equation*}
to a suitable rapid decay homology group, depending on $[\omega_{a,b}]$ and $k$, such that, for $i =1, 2$, the integrals 
\begin{equation*}
\int_{[\cL^b_i]} \Omega^b_{T, k},\,\int_{[\cL^b_i[1]]} \Omega^b_{T, k}
\end{equation*} 
are well defined, where we set 
\begin{equation*}
[\cL^b_i] = \Gamma(L_i),\,[\cL^b_i[1]] = \Gamma(L_i[1]).
\end{equation*}

\subsection{Exact triangle}\label{ExactSec} Suppose $p \in \left(-\tan\left(\frac{\pi}{n}\right), 0\right)$. Then there is a nontrivial morphism space
\begin{equation*}
\Hom(\cL^b_2, \cL^b_1) \cong \Hom(L_2, L_1) \cong H^0(X, \olo(- k a p H)).
\end{equation*}

By the geometric interpretation of the cone construction in terms of connected sums, it is natural to expect that for $b < 1$ the sLag multi-section $\cL^b$ should fit into an exact triangle 
\begin{equation*}
\cL^b_2 \xrightarrow{\varphi} \cL^b_1  \to \cL^b \cong \cone(\varphi) \to \cL_2[1]
\end{equation*}
for a suitable morphism $\varphi \in \Hom(\cL^b_2, \cL^b_1)$. This fits well with the fact that, by Lemma \ref{PhasesLem}, $\cL^b_1$, $\cL^b_2[1]$ have the same Lagrangian phase angle. 

Then the correct slope inequality predicted by the Thomas-Yau conjectures would be
\begin{equation*}
\arg \int_{[\cL^b_1]} \Omega^b_{T, k}  < \arg \int_{[\cL^b_2[1]]} \Omega^b_{T, k}, 
\end{equation*}
which also fits well with Theorem \ref{JoyceThm} and with our Theorem \ref{MainThmIntro}. Similarly, when $p \in \left(0, \tan\left(\frac{\pi}{n}\right)\right)$, there is a nontrivial morphism space
\begin{equation*}
\Hom(\cL^b_1, \cL^b_2) \cong \Hom(L_1, L_2) \cong H^0(X, \olo(k a p H)),
\end{equation*}
and for $b < 1$ the sLag $\cL^b$ should fit into an exact triangle 
\begin{equation*}
\cL^b_1 \xrightarrow{\varphi} \cL^b_2  \to \cL^b \cong \cone(\varphi) \to \cL_1[1]
\end{equation*}
for some $\varphi \in \Hom(\cL^b_2, \cL^b_1)$, so the correct Thomas-Yau slope inequality should be
\begin{equation*}
\arg \int_{[\cL^b_2]} \Omega^b_{T, k}  < \arg \int_{[\cL^b_1[1]]} \Omega^b_{T, k}. 
\end{equation*}

However, we note that a priori $\cL^b \subset M_{a, b}$ is only a multi-section in the differential-geometric model $M_{a, b}$, and so it is not obvious how to associate to it an object of $\FS(T, W_{a, b})$, or in what sense it represents the (well defined) object $\cone(\varphi)$ for some suitable $\varphi \in \Hom(\cL^b_2, \cL^b_1)$.
\section{Proofs of slope inequalities}\label{SlopeInequSec}
In the present Section we will complete the proofs of Theorems \ref{MainThmIntro}, \ref{BStabCorSurf} and \ref{BStabCor}. 

\begin{prop}\label{EquivProp}
Suppose $p \in \left(-\tan\left(\frac{\pi}{n}\right), 0\right)$. Then for admissible, sufficiently large $k > 0$ all $b$ sufficiently close to $1$, the condition  
\begin{equation*}
\arg \int_{[\cL^b_1]} \Omega^b_{T, k}  < \arg \int_{[\cL^b_2[1]]} \Omega^b_{T, k} 
\end{equation*}
is well defined and equivalent to the condition $b < 1$. Similarly, for $p \in \left(0, \tan\left(\frac{\pi}{n}\right), 0\right)$, the condition
\begin{equation*}
\arg \int_{[\cL^b_2]} \Omega^b_{T, k}  < \arg \int_{[\cL^b_1[1]]} \Omega^b_{T, k} 
\end{equation*}
is well defined and equivalent to the condition $b < 1$.
\end{prop}
We will first prove a preliminary result. Let $L$ denote a shifted $\R$-line bundle on $X$. We introduce the \emph{central charge}
\begin{align*}
Z(L) = \int_X (\omega_{a, b} + \ii c_1(L))^n
\end{align*}
and the corresponding \emph{$Z$-slope} 
\begin{equation*}
\lambda(L) = \frac{\Imm Z(L)}{\Rea Z(L)}.
\end{equation*}
We work with the $\R$-divisors
\begin{equation*}
L_1 = - a p H + q E,\, L_2 = q E.
\end{equation*}
\begin{prop}\label{SlopeInequProp} Suppose $p \in \left(-\tan\left(\frac{\pi}{n}\right), 0\right)$. For all $b$ sufficiently close to $1$, the condition $b < 1$ is equivalent to the slope inequality
\begin{equation*}
\lambda(L_1) < \lambda(L_2),
\end{equation*}
as well as to the phase inequality 
\begin{equation*}
\arg Z(L_1) < \arg Z(L_2[1]),
\end{equation*}
where $Z(L_1)$, $Z(L_2[1])$ lie in the upper half plane. Similarly, for $p \in \left(0, \tan\left(\frac{\pi}{n}\right)\right)$, the condition $b < 1$ is equivalent to the slope inequality
\begin{equation*}
\lambda(L_2) < \lambda(L_1),
\end{equation*}
as well as to the phase inequality 
\begin{equation*}
\arg Z(L_2) < \arg Z(L_1[1]) 
\end{equation*}
where $Z(L_2)$, $ Z(L_1[1])$ lie in the upper half plane.
\end{prop}
\begin{proof} It is enough to show the result for $p \in \left(-\tan\left(\frac{\pi}{n}\right), 0\right)$, the case when $p$ is positive being entirely analogous. 

Note that, by construction, we have
\begin{equation*}
\lambda(L_1)|_{b = 1} = \lambda(L_2)|_{b = 1},
\end{equation*}
so Proposition \ref{SlopeInequProp} follows if we can show 
\begin{equation*}
\del_b \lambda(L_1)|_{b = 1} > 0,\,\del_b \lambda(L_2)|_{b = 1} < 0.
\end{equation*}
We compute for $i =1, 2$
\begin{align*}
& \del_b \lambda(L_i)|_{b = 1} = \left(\frac{\del_b \Imm Z(L_i)}{\Rea Z(L_i)} - \frac{(\del_b \Rea Z(L_i)) \Imm Z(L_i)}{(\Rea Z(L_i))^2}\right)|_{b = 1}\\
& = \left(\frac{\Imm \del_b  Z(L_i)}{\Rea Z(L_i)} - \frac{(\Rea \del_b  Z(L_i)) \Imm Z(L_i)}{(\Rea Z(L_i))^2}\right)|_{b=1}.
\end{align*}
We have
\begin{align*}
& \del_b Z(L_i)|_{b = 1} = \int_X n (\omega_{a, b} + \ii c_1(L_i))^{n-1} \cdot ( - [E])\\
& = - n \int_E (\omega_{a, 1} + \ii c_1(L_i))^{n-1} = - n\int_E ( - [E] + \ii q [E])^{n-1} \\
& = - n (-1)^{n-1} (-1 + \ii q)^{n-1}, 
\end{align*}
so 
\begin{align*}
& \del_b \lambda(L_i)|_{b = 1}\\
& = \left(\frac{n (-1)^{n} \Imm(-1 + \ii q)^{n-1}}{\Rea Z(L_i)} - \frac{(n (-1)^{n} \Rea(-1 + \ii q)^{n-1}) \Imm Z(L_i)}{(\Rea Z(L_i))^2}\right)\big|_{b=1}.
\end{align*}
Recall 
\begin{align*}
q = \tan\left( \frac{\hat{\theta} - \frac{\pi}{2}}{n-1} \right).
\end{align*}
Then an elementary study shows that, in our current range
\begin{equation*}
p \in \left(-\tan\left(\frac{\pi}{n}\right), 0\right),\,\hat{\theta} \in \left(0, \frac{\pi}{2}\right),
\end{equation*}
we have
\begin{align*}
& (-1)^{n} \Imm(-1 + \ii q)^{n-1} < 0,\,(-1)^{n} \Rea(-1 + \ii q)^{n-1} < 0. 
\end{align*}
We study the signs of $\Imm Z(L_i)|_{b=1}$, $\Rea Z(L_i)|_{b=1}$ for $i = 1, 2$. According to our computation in the previous Section, we have
\begin{equation*} 
Z(L_1)|_{b = 1} \in \R_{> 0 } e^{\ii (-\hat{\theta} + \pi)},\,Z(L_2)|_{b = 1}\in \R_{> 0 } e^{- \ii \hat{\theta}}.
\end{equation*}
Thus 
\begin{equation*}
\Rea Z(L_1)|_{b = 1} \in - \R_{> 0 } \cos(\hat{\theta}) < 0,\,\Imm Z(L_1)|_{b = 1} \in \R_{> 0 } \sin(\hat{\theta}) > 0, 
\end{equation*}
and similarly
\begin{equation*}
\Rea Z(L_2)|_{b = 1} \in \R_{> 0 } \cos(\hat{\theta}) > 0,\,\Imm Z(L_2)|_{b = 1} \in - \R_{> 0 } \sin(\hat{\theta}) < 0. 
\end{equation*}
It follows that
\begin{align*}
& \frac{n (-1)^{n} \Imm(-1 + \ii q)^{n-1}}{\Rea Z(L_1)} \big|_{b=1}> 0,\\
& \frac{(n (-1)^{n} \Rea(-1 + \ii q)^{n-1}) \Imm Z(L_1)}{(\Rea Z(L_1))^2}\big|_{b = 1} < 0  
\end{align*}
so 
\begin{align*}
& \del_b \lambda(L_1)|_{b = 1}\\
& = \left(\frac{n (-1)^{n} \Imm(-1 + \ii q)^{n-1}}{\Rea Z(L_1)} - \frac{(n (-1)^{n} \Rea(-1 + \ii q)^{n-1}) \Imm Z(L_1)}{(\Rea Z(L_1))^2}\right)\big|_{b=1}\\
& > 0.
\end{align*}
Similarly we have
\begin{align*}
& \frac{n (-1)^{n} \Imm(-1 + \ii q)^{n-1}}{\Rea Z(L_2)} \big|_{b=1} < 0,\\
& \frac{(n (-1)^{n} \Rea(-1 + \ii q)^{n-1}) \Imm Z(L_2)}{(\Rea Z(L_2))^2}\big|_{b = 1}  > 0  
\end{align*}
so 
\begin{align*}
& \del_b \lambda(L_2)|_{b = 1}\\
& = \left(\frac{n (-1)^{n} \Imm(-1 + \ii q)^{n-1}}{\Rea Z(L_2)} - \frac{(n (-1)^{n} \Rea(-1 + \ii q)^{n-1}) \Imm Z(L_2)}{(\Rea Z(L_2))^2}\right)\big|_{b=1}\\
& < 0. 
\end{align*}
Note that our computations show that in fact $Z(L_1), Z(L_2[1]) = -Z(L_2)$ lie in the upper half plane, so the condition $\lambda(Z_1) < \lambda(Z_2)$ is equivalent to 
\begin{equation*}
\arg Z(L_1) < \arg Z(L_2[1]).
\end{equation*}

This completes the proof of Proposition \ref{SlopeInequProp}.
\end{proof}
We can now prove Proposition \ref{EquivProp}.
\begin{proof}[Proof of Proposition \ref{EquivProp}] Suppose $L$ is any shifted line bundle. A direct computation shows that  
\begin{equation*}
Z(L) = - e^{\ii (n-2)\frac{\pi}{2}} \int_{X} e^{-\ii \omega} \ch(L).
\end{equation*}
As observed, in a more general context, in \cite{J_toricThomasYau}, Section 5.2 (see in particular $(5.4)$ there), there is an expansion 
\begin{equation*}
\int_{X} e^{-\ii k \omega} \ch(L^{\otimes k}) = \left(\frac{1}{2\pi \ii}\right)^n k^{-n} \int_{\Gamma(L^{\otimes k})} e^{-W(k \omega)} \Omega_0\left(1 + O(k^{-1})\right), 
\end{equation*}
where
\begin{equation*}
\Gamma\!: K^0(X) \to H_n(T, \{\Rea W(k \omega) \gg 0\}; \Z)
\end{equation*}
denotes a homomorphism to a suitable rapid decay homology group (depending on $[\omega]$), as in Section \ref{IntegrationSec}. This follows from results of Iritani, see e.g. \cite{Iritani_survey}, Section 7 for an exposition. As a consequence we have an expansion
\begin{equation*}
\arg Z(L) = \arg \int_{\Gamma(L^{\otimes k})} e^{-W(k \omega)} \Omega_0\left(1 + O(k^{-1})\right).
\end{equation*}
We may replace $\omega$ and $L_1$, $L_2$ by a suitable common rescaling, and applying the identity above to these scalings gives the required result (since the left hand side is invariant), where the natural integration cycles are defined, as in Section \ref{IntegrationSec}, by $[\cL^b_i] = \Gamma(L_i)$, respectively $[\cL^b_i[1]] = \Gamma(L_i[1])$ (see \cite{Fang_charges} for naturality, i.e. compatibility with the LYZ transform of line bundles). 
\end{proof}
\begin{proof}[Proof of Theorem \ref{MainThmIntro}]
Assume that $p \in \left(-\tan\left(\frac{\pi}{n}\right), 0\right)$. Then we define the pair of graded sLags $\chL^s_i$ appearing in   Theorem \ref{MainThmIntro} as
\begin{equation*}
(\chL^s_1, \chL^s_2) := (\cL^{1-s}_2[1], \cL^{1-s}_1).
\end{equation*}
Similarly, for $p \in \left(0, \tan\left(\frac{\pi}{n}\right), 0\right)$, we set
\begin{equation*}
(\chL^s_1, \chL^s_2) := (\cL^{1-s}_1[1], \cL^{1-s}_2).
\end{equation*}

By Section \ref{ExactSec}, the group $HF^1(\chL^s_1, \chL^s_2)$ is nontrivial. Thus, by Lemma \ref{PropertiesLem}, the claims of Theorem \ref{MainThmIntro} follow immediately from Proposition \ref{EquivProp}. 
\end{proof}
We can also complete the proofs Theorems \ref{MainThmIntro}, \ref{BStabCorSurf} and \ref{BStabCor}.
\begin{proof}[Proof of Theorem \ref{BStabCorSurf}] By construction, our sLag sections $\chL^s_1$, $\chL^s_2$ are LYZ transforms of Hermitian structures on genuine line bundles (scalings of the $\R$-line bundles $L_1$, $L_2$), and so give well-defined objects of $\FS(T, W^s) \cong D^b(X)$. The result then follows straightforwardly if we can show that suitable shifts of $\chL^s_1$, $\chL^s_2$ lie in the heart $\check{\cA}^s$ of the stability condition $(\check{\cA}^s, \check{Z}^{s})$ on $\FS(T, W^s)$ induced from $(\cA, Z)$ on $D^b(X)$, and that they are $(\check{\cA}^s, \check{Z}^{s})$-stable. 

We first consider the signs of the degrees of the $\R$-line bundles $L_1$, $L_2$ with respect to $\omega_{a, b}$, for $b$ sufficiently close to $1$. We have
\begin{align*}
& L_1 \cdot \omega_{a, b} = (-a p H + q E)\cdot (a H - b E) = -a^2 p + b q,\\
& L_2 \cdot \omega_{a, b} = (q E)\cdot (a H - b E) = b q. 
\end{align*}
According to Example \ref{SurfExm}, we have
\begin{equation*}
p = -2 \cot(\hat{\theta}),\,a = \frac{1}{\sin^2(\hat{\theta})},\,q = \tan\left( \hat{\theta} - \frac{\pi}{2} \right) = -\cot(\hat{\theta}),
\end{equation*} 
so
\begin{equation*}
L_1 \cdot \omega_{a, b} = \cot(\hat{\theta})\left(\frac{2}{\sin^4(\hat{\theta})}-b\right),\,L_2 \cdot \omega_{a, b} = - b\cot(\hat{\theta}). 
\end{equation*} 

Suppose $p \in (-\tan(\frac{\pi}{2}), 0 ) = (-\infty, 0)$, so $\hat{\theta} = (0, \frac{\pi}{2})$. Then the above expressions show that, for $b$ sufficiently close to $1$, we have
\begin{align*}
L_1 \cdot \omega_{a, b} > 0,\,L_2 \cdot \omega_{a, b} < 0. 
\end{align*}
Since the heart $\cA$ is defined in \cite{ArcaraBertram} as the tilting of $\Coh(X)$ at the torsion pair $(\Coh^{>0}(X), \Coh^{\leq}(X))$, it follows that in this case $L_1,\,L_2[1] \in \cA$, and since in this range we defined $(\chL^s_1, \chL^s_2) := (\cL^{1-s}_2[1], \cL^{1-s}_1)$, it follows that $\chL^s_1,\,\chL^s_2 \in \check{\cA}^s$.

Similarly, for $p \in (0, \tan(\frac{\pi}{2}) ) = (0, \infty)$, we have $\hat{\theta} = (\frac{\pi}{2}, \pi)$, we have $\cot(\hat{\theta}) < 0$ and so
\begin{align*}
L_1 \cdot \omega_{a, b} < 0,\,L_2 \cdot \omega_{a, b} > 0, 
\end{align*}  
from which we get $L_1[1],\,L_2 \in \cA$, and, since $(\chL^s_1, \chL^s_2) := (\cL^{1-s}_1[1], \cL^{1-s}_2)$ in this range, we still conclude $\chL^s_1,\,\chL^s_2 \in \check{\cA}^s$.

Given these inclusions, the claim that $\chL^s_1,\,\chL^s_2$ are $(\check{\cA}^s, \check{Z}^{s})$-stable follows from the fact that $L_1,\,L_2[1] \in \cA$, respectively $L_1[1],\,L_2 \in \cA$ are $(\cA, Z)$-stable, since the dHYM equation is solvable on $L_1$, $L_2$ and by \cite{J_toricThomasYau}, Theorem 1.10 (which in turn follows from a result of Collins-Shi \cite{Collins_stability}).
\end{proof}
\begin{proof}[Proof of Theorem \ref{BStabCor}] As above, we may replace $L_1$, $L_2$ by suitable scalings, such that the corresponding $(\cL^{1-s}_2[1], \cL^{1-s}_1)$, respectively $(\cL^{1-s}_1[1], \cL^{1-s}_2)$ are well-defined pairs of objects of $\FS(T, W^s) \cong D^b(X)$, showing claim $(i)$.  

Suppose $p \in \left(-\tan\left(\frac{\pi}{n}\right), 0\right)$, so that the corresponding phase angle satisfies $\hat{\theta} \in \left(0, \frac{\pi}{2}\right)$ and we have $q < 0$. We compute 
\begin{align*}
& \check{Z}^{s, \Gamma}(k^{-1}\chL^s_1) = Z(L_2[1]) + \Gamma \cdot \ch_1(L_2[1])\\
& = Z(L_2[1]) - \left(\frac{H^2 + E^2}{6} + C_0 \omega^2_{a,1}\right)\cdot(q E),
\end{align*}
and similarly
\begin{align*}
& \check{Z}^{s, \Gamma}(k^{-1}\chL^s_2) = Z(L_1) + \Gamma \cdot \ch_1(L_1)\\
& = Z(L_1) + \left(\frac{H^2 + E^2}{6} + C_0 \omega^2_{a,1}\right)\cdot(- a p H + q E).
\end{align*}
Thus
\begin{align*}
& \arg \check{Z}^{s, \Gamma}(k^{-1}\chL^s_1) = \arg\left( \frac{Z(L_2[1])}{- a p} + \left(\frac{H^2 + E^2}{6} + C_0 \omega^2_{a,1}\right)\cdot(\frac{q}{a p} E)\right),\\
& \arg \check{Z}^{s, \Gamma}(k^{-1}\chL^s_2) = \arg\left( \frac{Z(L_1)}{-a p} + \left(\frac{H^2 + E^2}{6} + C_0 \omega^2_{a,1}\right)\cdot(H - \frac{q}{a p} E)\right).
\end{align*}
By our computations in Section \ref{ScalingSec}, we have
\begin{equation*}
\lim_{p \to - \tan\left(\frac{\pi}{n}\right)^+} \frac{q }{a p} = 0, 
\end{equation*}
so, for $\hat{\theta} \in \left(0, \frac{\pi}{2}\right)$ sufficiently close to $0$, it is enough to show that 
\begin{equation*}
\frac{Z(L_2[1])}{- a p},\,\frac{Z(L_1)}{-a p} + \left(\frac{H^2 + E^2}{6} + C_0 \omega^2_{a,1}\right)\cdot H 
\end{equation*}
lie in the upper half plane and satisfy
\begin{equation*}
\arg \frac{Z(L_2[1])}{- a p} > \arg\left(\frac{Z(L_1)}{-a p} + \left(\frac{H^2 + E^2}{6} + C_0 \omega^2_{a,1}\right)\cdot H\right). 
\end{equation*}
By Proposition \ref{SlopeInequProp}, we know that  
\begin{equation*}
\frac{Z(L_2[1])}{- a p},\,\frac{Z(L_1)}{-a p} 
\end{equation*}
lie in the upper half plane and satisfy
\begin{equation*}
\arg \frac{Z(L_2[1])}{- a p} > \arg \frac{Z(L_1)}{-a p}, 
\end{equation*}
so it is enough to show that we have
\begin{equation*}
\left(\frac{H^2 + E^2}{6} + C_0 \omega^2_{a,1}\right)\cdot H  > 0. 
\end{equation*}
According to \cite{BernardaraMacri_threefolds}, Section 4.B, the constant $C_0$ is given explicitly by 
\begin{equation*}
C_0 = \frac{10 \sqrt{30}}{1323}-\frac{3}{98} \approx 0.01 > 0, 
\end{equation*} 
so
\begin{equation*}
\left(\frac{H^2 + E^2}{6} + C_0 \omega^2_{a,1}\right)\cdot H = \frac{1}{6} + C_0 a^2 > 0. 
\end{equation*}

On the other hand, for $p \in \left(0, \tan\left(\frac{\pi}{n}\right)\right)$, the corresponding phase angle satisfies $\hat{\theta} \in \left(\frac{\pi}{2}, \pi \right)$, and we have $q > 0$. We compute
\begin{align*}
& \check{Z}^{s, \Gamma}(k^{-1}\chL^s_1) = Z(L_1[1]) - \left(\frac{H^2 + E^2}{6} + C_0 \omega^2_{a,1}\right)\cdot(- a p H + q E),\\
& \check{Z}^{s, \Gamma}(k^{-1}\chL^s_2) = Z(L_2) + \left(\frac{H^2 + E^2}{6} + C_0 \omega^2_{a,1}\right)\cdot(q E),
\end{align*}
so
\begin{align*}
& \arg \check{Z}^{s, \Gamma}(k^{-1}\chL^s_1) = \arg\left( \frac{Z(L_1[1])}{q} - \left(\frac{H^2 + E^2}{6} + C_0 \omega^2_{a,1}\right)\cdot(-\frac{a p}{q}H + E)\right),\\
& \arg \check{Z}^{s, \Gamma}(k^{-1}\chL^s_2) = \arg\left( \frac{Z(L_2)}{q} + \left(\frac{H^2 + E^2}{6} + C_0 \omega^2_{a,1}\right)\cdot E \right).
\end{align*}
According to Section \ref{ScalingSec}, we have
\begin{equation*}
\lim_{p \to 0} \frac{a(p) p}{q(p)} = 0,
\end{equation*}
so, for $\hat{\theta} \in \left(\frac{\pi}{2}, \pi \right)$ sufficiently close to $\frac{\pi}{2}$, it is enough to show that 
\begin{equation*}
\frac{Z(L_1[1])}{q} - \left(\frac{H^2 + E^2}{6} + C_0 \omega^2_{a,1}\right)\cdot E,\,\frac{Z(L_2)}{q} + \left(\frac{H^2 + E^2}{6} + C_0 \omega^2_{a,1}\right)\cdot E 
\end{equation*}
lie in the upper half plane and satisfy
\begin{align*}
& \arg\left(\frac{Z(L_1[1])}{q} - \left(\frac{H^2 + E^2}{6} + C_0 \omega^2_{a,1}\right)\cdot E \right)\\
&>\arg\left(\frac{Z(L_2)}{q} + \left(\frac{H^2 + E^2}{6} + C_0 \omega^2_{a,1}\right)\cdot E \right).
\end{align*}
By Proposition \ref{SlopeInequProp}, we know that  
\begin{equation*}
\frac{Z(L_1[1])}{q},\,\frac{Z(L_2)}{q}   
\end{equation*}
lie in the upper half plane and satisfy
\begin{equation*}
\arg \frac{Z(L_1[1])}{q} > \arg \frac{Z(L_2)}{q}, 
\end{equation*}
so it is enough to show that we have
\begin{equation*}
\left(\frac{H^2 + E^2}{6} + C_0 \omega^2_{a,1}\right)\cdot E  > 0, 
\end{equation*}
which follows easily from $C_0 > 0$.
\end{proof}
\section{Momentum mean curvature flow}\label{FlowSec}

We follow the notation of Section \ref{CalabiSec}. Let us denote the eigenvalues of the endomorphism $\omega^{-1} \alpha$ by $\{\lambda_1, \cdots, \lambda_n\}$. Consider a family of representatives $\alpha_t$ of $[\alpha]$ parametrised as $\alpha_t = \alpha_0 + \ii \del\delbar \phi_t$. The general \emph{line bundle mean curvature flow (MCF)}, introduced in \cite{JacobYau_special_Lag}, is the parabolic flow of potentials  
\begin{equation*}
\dot{\phi}_t = \sum^n_{i = 1} \arctan(\lambda_i) - \hat{\theta}.
\end{equation*}
Fixed points of the line bundle MCF have constant phase angle, equal to $\hat{\theta}$.

We specialise to the case when $X = \Bl_p \PP^n$. Chan and Jacob \cite{ChanJacob} showed that the line bundle MCF with Calabi ansatz on $X$ is given in terms of the momentum coordinate $x = u'(\log |z|^2)$ by
\begin{equation}\label{lbMCF}
\dot{v} = (n-1)\arctan\left(\frac{f}{x}\right) + \arctan(f') - \hat{\theta}.
\end{equation}
By taking the derivative of \eqref{lbMCF} with respect to $\rho = \log |z|^2$, they obtained a flow for \emph{momentum profiles} $f(x)$ (rather than \emph{K\"ahler potentials}) given by
\begin{equation}\label{profileMCF}
\dot{f} = u''(x)\left( \frac{f''}{1+(f')^2} + (n-1)\frac{x f' - f}{x+f^2}\right),
\end{equation}
where $\dot{f} = \del_t f$, $f' = \del_{x} f$, $u'' = \del^2_{\rho, \rho} u$. Note that, since the symplectic potential satisfies $u''(1) = u''(a) = 0$, the flow preserves the boundary conditions, i.e. it preserves the cohomology class. As the flow \eqref{profileMCF} is obtained as the derivative of \eqref{lbMCF}, its fixed points have constant phase angle, not prescribed a priori. This is especially important for us, since the Lagrangian phase angle will only be locally constant in our application.

As observed by Chan and Jacob, \eqref{profileMCF} can be extended to a flow of parametric curves $\gamma_t(s) = (x_t(s), y_t(s))\! : [0, 1] \to [1, a] \times \R$, $s$ denoting arc-length, defined in terms of the quantities
\begin{equation*}
\kappa_t = \frac{d}{d s} \arctan\left(\frac{y'_t}{x'_t}\right),\,\xi_t = \frac{d}{ds} \arctan\left(\frac{y_t}{x_t}\right),\,\opN_t = e^{\ii \frac{\pi}{2}} \gamma'_t
\end{equation*}
as 
\begin{equation}\label{MCF}
\dot{\gamma}_t = u''(x_t)\left(\kappa_t + (n-1)\xi_t \right)\opN_t. 
\end{equation}
Note that $\kappa_t$ is the usual plane curvature, so in this notation the curve-shortening flow is $\dot{\gamma}_t =  \kappa_t \opN_t$. 

We will call \eqref{MCF} the \emph{momentum mean curvature flow (MCF)}. 
\begin{proof}[Proof of Theorem \ref{FlowCor}.] Fix $b > 1$. For the sake of the proof, we denote the given sLag sections by $\chL^b_i$, while $s$ denotes the arc-length parameter as above. 

Suppose $\chL_t \subset M_{a, b}$ is a family of Calabi-symmetric Lagrangian multi-sections, with a single critical point, asymptotic to $\chL^b_1 \cup \chL^b_2$ at infinity in $M_{a, b}$ and evolving by the momentum MCF. By our discussion above $\chL_t$ is obtained as the combination of the LYZ and Legendre transforms of a family of curves $\gamma_t(s) = (x_t, y_t)\! : [0, 1] \to [b, a] \times \R$, with endpoints $(a, 0)$, $(a, p)$, evolving by the flow \eqref{MCF}, which are multi-sections for the projection $[b, a] \times \R \to [b, a]$ with a single critical point. We denote the latter by $\gamma_t(s_t)$ and consider its evolution. We choose an orientation for $\gamma_t$ so that $\gamma'_t > 0$ (the result is invariant under this choice).

By criticality, $\gamma_t$ has vertical tangent at $\gamma_t(s_t)$, and, with our choice of orientation, by the usual characterisation of $\kappa(s)$ as the radius of an obsculating circle, we obtain $\kappa_t(s_t) > 0$. Similarly we compute
\begin{align*}
\xi_t(s_t) = \frac{d}{ds} \arctan\left(\frac{y_t}{x_t}\right)\big|_{s_t} = \frac{x_t y'_t - y_t  x'_t }{x^2_t+y^2_t}\big|_{s_t} = \frac{x_t y'_t}{x^2_t+y^2_t}\big|_{s_t} > 0,
\end{align*} 
noting that $x'(s_t) = 0$ since $x(s)$ achieves a minimum at $s_t$, while $x_t, y'_t > 0$. On the other hand, we have $u''(x(s_t)) > 0$ since the critical point lies in the interior $(b, a) \times \R$. The flow \eqref{MCF} then implies that, at the critical point, $\dot{\gamma}_t(s_t)$ is a strictly positive multiple of $\opN_t(s_t) = (-1,0)$.  

It follows that if $\chL_t$ exists for all times and converges locally smoothly on $M_{a, b}$, so that the same holds for $\gamma_t(s) = (x_t, y_t)\! : [0, 1] \to [b, a] \times \R$, then the family of critical points must satisfy $\lim_{t \to \infty} x_t(s_t) = x_{\infty} \in [b, a)$, and $\gamma_t$ converges to the union of two graphs of profiles $f_i$ defined on $[x_{\infty}, a)$, each of which is a fixed point of the flow \eqref{profileMCF} with fixed phase angle $\hat{\theta}_i$ and one fixed boundary condition $f_1(a) = p$, $f_2(a) = 0$. 

If we have $x_{\infty} > b$, then, by the smooth convergence of \eqref{profileMCF}, the union $\chL_{\infty, 1} \cup \chL_{\infty, 1}$ of the limit Lagrangians $\chL_{\infty, i}$ corresponding to the profiles $f_i \in C^{\infty}[x_{\infty}, a)$ gives a connected, embedded Lagrangian $\chL_{\infty}$, with phase angle $\hat{\theta}_{\infty} = \hat{\theta}_1 = \hat{\theta}_2$, with Calabi symmetry, satisfying the boundary conditions $f_1(a) = p$, $f_2(a) = 0$. But then $\chL_{\infty}$ must coincide with the Calabi-symmetric sLag multi-section $\cL^b$ constructed in Section \ref{MultiSec}, which is a contradiction since we are in the unstable case $b > 1$, when $\cL^b$ is not contained in $M_{a, b}$.

Thus, we must have $x_{\infty} = b$. Let $\tilde{f}_i$, $i =1, 2$ denote the momentum profiles of the sLag sections $\cL^b_i$ constructed in Section \ref{sLagsSec}. The profiles $\tilde{f}_i$ are solutions of the parabolic flow \eqref{profileMCF} and, by the maximum principle, they act as barriers, so that, if $\chL$ is contained in the region bounded by $\chL^b_1$, $\chL^b_2$, the evolution of $\gamma_t(s)$ is confined to the convex region in $[b, a] \times \R$ delimited by the graphs of $\tilde{f}_i$. This implies that we must have $(x_{\infty}, y_{\infty}) = (b, q)$, so $\gamma_t(s)$ converges locally smoothly to the graphs of $\tilde{f}_i$, up to a constant shift, and by comparing phase angles we see that $\chL_{\infty}$ is given by the disconnected union $\chL^b_1 \cup \chL^b_2$ as claimed.  
\end{proof}
\section{Split Fano bundles}\label{BundlesSec}
We consider the dHYM equation with Calabi ansatz on $X_{r, m}$ with respect to $\omega$ and an auxiliary $(1,1)$-class
\begin{equation*}
[\alpha] \in \xi_2 [D_H] + q [D_{\infty}].
\end{equation*}
By the results of Jacob \cite{Jacob_Harmonic}, Section 6, similarly to the case recalled in our Section \ref{CalabiSec}, solutions correspond to graphical portions of the level sets of the harmonic polynomial $\Imm w(z)$, where $z = x + \ii y$, 
\begin{equation*}
w(z) = e^{-\ii\hat{\theta}}\int^z_0 (\xi + s)^m s^r ds
\end{equation*}
for
\begin{equation*}
\xi = \xi_1 + \ii \xi_2,\,\int_X(\omega + \ii \alpha)^{m+r+1} \in \R_{>0} e^{\ii \hat{\theta}}.
\end{equation*}

More precisely, let $\cC_0$ denote the locus $\{\Imm w(z) = 0\}$. Then dHYM solutions correspond to functions $f$ defined on a \emph{modified} momentum interval $[0, b]$, such that $f([0, b]) \subset \cC_0$, with boundary conditions
\begin{equation*}
f(0) = 0,\,f(b) = q.
\end{equation*}

Note that, unlike the approach for $X = \Bl_p \PP^n \cong X_{0, n - 1}$ described in Section \ref{CalabiSec}, now the harmonic polynomial $\Imm w(z)$ depends on the K\"ahler parameter $\xi_1 > 0$. 

We can absorb this dependence in the interval of definition by the change of variables $z = \xi_1 \tilde{z}$. Set 
\begin{equation*}
\tilde{w}(\tilde{z}) = e^{-\ii\hat{\theta}}\int^{\tilde{z}}_0 (1 + \ii \eta + s)^m s^r ds,\,\eta\in \R,
\end{equation*} 
and let $\tilde{\cC}_0$ denote the locus $\{\Imm \tilde{w}(\tilde{z}) = 0\}$. We consider functions $\tilde{f}$, defined on $[0, \xi^{-1}_1 b]$, such that $\tilde{f}([0, \xi^{-1}_1 b]) \subset \tilde{\cC}_0$, satisfying the boundary conditions
\begin{equation*}
\tilde{f}(0) = 0,\,\tilde{f}(\xi^{-1}_1 b) = q.
\end{equation*}
By our discussion, there is a bijective correspondence between functions $\tilde{f}$ as above and dHYM solutions with Calabi ansatz with cohomology class  
\begin{equation*}
[\alpha] \in (\xi_1 \eta)[D_H] + (\xi_1 q) [D_{\infty}],
\end{equation*}
with respect to a Calabi ansatz metric in the class 
\begin{equation*}
[\omega_{\xi_1, b}] \in \xi_1 [D_H] + b [D_{\infty}].
\end{equation*}

We note that for $\xi'_1 < \xi_1$ there is an inclusion
\begin{equation*}
[0, \xi^{-1}_1 b] \subset [0, (\xi'_1)^{-1} b] 
\end{equation*}
corresponding to an inclusion of momentum polytopes
\begin{equation*}
\Delta(\omega_{\xi_1, b}) \subset \Delta(\omega_{\xi'_1, b}). 
\end{equation*}
Thus, for fixed $\eta \in \R$, as the K\"ahler parameter $\xi_1 > 0$ varies, a connected subset $\tilde{\cC}'_0 \subset \tilde{\cC}_0$ contained in $(0, \xi^{-1}_1 b) \times \R$ yields a special Lagrangian multi-section $\cL^{\xi'_1}$ of the almost Calabi-Yau manifold $M_{\xi, b} \to \Delta^o(\omega_{\xi_1, b})$, defined as the closure of the union of the sLag sections which are LYZ transforms of its branches with respect to $\omega_{\xi_1, b}$.  

Given this correspondence, we go back to the more convenient coordinate $z$. By construction the locus $\{\Imm w(z) = 0\}$ near $z = 0$ is the union of several analytic branches $\cB_1,\ldots, \cB_{\ell}$; we single out a specific branch $\cB$ by imposing that it has vertical tangent at $z = 0$. We have a series expansion
\begin{equation*}
w(\ii t) = \frac{ \ii^{r + 1} \xi^m e^{-\ii \hat{\theta}}}{r+1} t^{r+1}+ O(t^{r+2}).
\end{equation*}
So writing $\xi = \rho e^{\ii \psi}$ the condition that an analytic branch of $\{\Imm w(z) = 0\}$ has vertical tangent at $z = 0$ is given by 
\begin{equation*}
\hat{\theta} = m \psi + \frac{r + 1}{2} \pi \mod \pi \Z.
\end{equation*}
Figure \ref{fig:FanoBundlesBranches} shows the configuration for $m = r = 2$, $\xi = 2 - \ii$, for which $\hat{\theta} = \arctan(\frac{3}{4}) + h \pi$.  
\begin{figure}[h]
  \includegraphics[width=5cm]{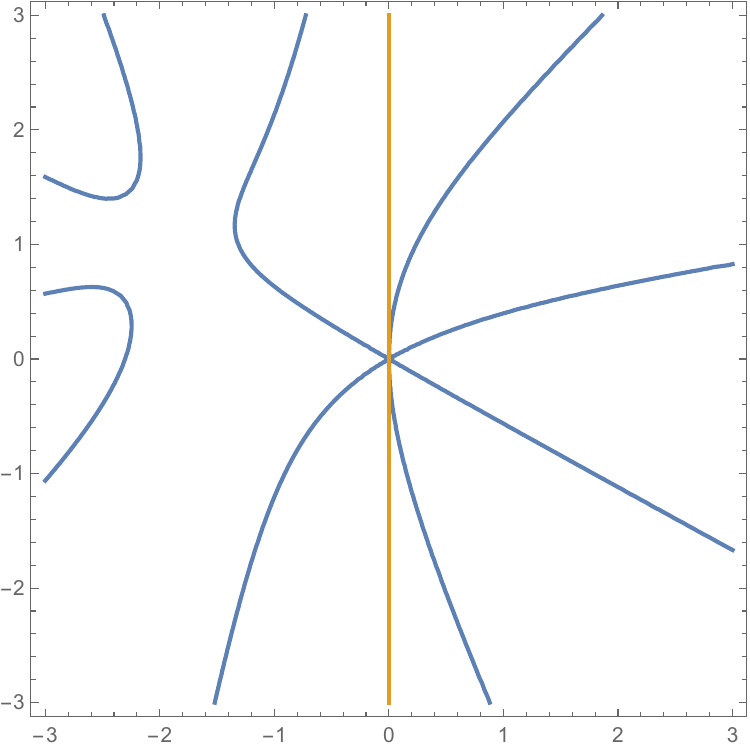}
  \caption{Selecting a branch $\cB \subset \{\Imm w(z) = 0\}$.}
  \label{fig:FanoBundlesBranches}
\end{figure}
We need to choose our configuration suitably so that the branches of $\cB \cap [0, b] \times \R$ correspond to line bundles $\cL_1$, $\cL_2$, up to scaling by $k > 0$. Writing 
\begin{equation*}
\cB \cap \{x = b\} = \{b + \ii q,\,b + \ii q'\},
\end{equation*}
the latter condition holds iff the cohomological parameters $\xi_2$, $q$, $q'$ are commensurable, i.e. generically iff
\begin{equation*}
\frac{q}{\xi_2},\,\frac{q}{q'} \in \Q.
\end{equation*}
Note that for, fixed $\hat{\theta}$, the quantities $\frac{q}{\xi_2}$, $\frac{q}{q'}$ are generically well defined, locally surjective smooth functions of the K\"ahler parameters $\xi_1$, $b$. Therefore the rationality condition must hold for a dense subset of the set of K\"ahler parameters $(\xi_1, b) \in \R_{>0} \times \R_{>0}$ (and so for a dense subset of Lagrangian phase angles $\hat{\theta}$).

We can now complete the proof of Theorem \ref{BundlesThmIntro}.
\begin{proof}[Proof of Theorem \ref{BundlesThmIntro}] Abusing notation slightly, we still write $\cL^{\xi'_1}$ for the sLag multi-section of $M_{k \xi'_1, k b}$ corresponding to the analytic branch $\cB$ for all $0 < k \xi'_1 < k \xi_1$, where as usual $k > 0$ is a scaling parameter, replacing $\omega, \alpha$ with $k\omega, k\alpha$.

Let us fix our notation so that $q > 0$, $q' < 0$. Introduce the $\R$-line bundles
\begin{equation*}
L_1 = \olo(k\xi_2 D_{H} + k q D_{\infty}),\,L_2 = \olo(k \xi_2 D_{H} + k q' D_{\infty}).
\end{equation*}
By our discussion above, we can choose $k > 0$ such that $L_1$, $L_2$ are genuine line bundles. By the same argument as in Section \ref{sLagsSec}, the line bundles $L_i$ admit unique solutions of the dHYM equation with respect to $k \omega_{\xi'_1, b}$ for all $\xi'_1$ sufficiently close to our reference parameter $\xi_1$. This is the claim $(i)$ in our statement.

We define the sLag sections $\cL^{\xi'_1}_i$ as the LYZ mirrors of $L_i$ with respect to these hermitian metrics.

Let us compute the Lagrangian phase angles of $\cL^{\xi'_1}_i$. According to \cite{Jacob_Harmonic}, Section 6, the lifted angle $\vartheta$ of a dHYM solution is given in terms of its momentum profile $f$ as
\begin{equation*}
\vartheta = m \arctan\left(\frac{\xi_2 + f}{\xi_1 + x}\right) + r \arctan\left(\frac{f}{x}\right) + \arctan{\frac{df}{dx}}.
\end{equation*} 
Applying this in our case, we see that the Lagrangian phase angles $\vartheta_i|_{\xi_1}$ of $\cL_i$ at $\xi_1$ are given by 
\begin{equation*}
\vartheta_1|_{\xi_1} = m\psi + \frac{r +1}{2} \pi,\,\vartheta_2|_{\xi_1} = m \psi - \frac{r +1}{2} \pi = m\psi +\frac{r +1}{2} \pi - (r+1)\pi. 
\end{equation*}
We choose our K\"ahler parameter $\xi_1$ generically so that the corresponding phase angle satisfies 
\begin{equation*}
\vartheta_1|_{\xi_1} = m\psi + \frac{r +1}{2} \pi \neq (2 r' + 1) \pi,\, r' \in \Z.
\end{equation*}
Then we have 
\begin{equation*}
m\psi + \frac{r+1}{2} \pi \in (-\pi, 0) \mod 2\pi, \textrm{ or } m\psi + \frac{r+1}{2} \pi \in (0, \pi) \mod 2\pi.
\end{equation*}
We claim that, for $b$ sufficiently small and $\xi'_1$ sufficiently close to $\xi_1$, the condition $\xi'_1 < \xi_1$ is equivalent to the phase inequality
\begin{equation*}
\arg \int_{[\cL^{\xi'_1}_1]} \Omega^b_{T, k}  < \arg \int_{[\cL^{\xi'_1}_2[1]]} \Omega^b_{T, k} 
\end{equation*}
in the first case, and to 
\begin{equation*}
\arg \int_{[\cL^{\xi'_1}_2]} \Omega^b_{T, k}  < \arg \int_{[\cL^{\xi'_1}_1[1]]} \Omega^b_{T, k} 
\end{equation*} 
in the second case. Clearly, this would prove $(ii)$ in our statement by choosing
\begin{equation*}
(\chL^s_1, \chL^s_2) := (\cL^{\xi_1+s}_2[1], \cL^{\xi_1+s}_1),\,(\chL^s_1, \chL^s_2) := (\cL^{\xi_1+s}_1[1], \cL^{\xi_1+s}_2)
\end{equation*}
respectively in the two cases.

We proceed as in Section \ref{SlopeInequSec}. Let us set
\begin{align*}
Z(L) = \int_{X_{r,m}} (\omega_{\xi_1, b} + \ii c_1(L))^{m+r+1}.
\end{align*}

Suppose $L$ is any line bundle with $c_1(L) = \xi_2 [D_{H}] + q_L [D_{\infty}]$. According to \cite{Jacob_Harmonic}, Section 6 (proof of Lemma 12), there is an identity
\begin{equation*}
Z(L) = \int_{X_{r,m}} (\omega_{\xi_1, b} + \ii c_1(L))^{m+r+1} = \frac{1}{2} \vol(S_{m, r}, \omega_{\xi_1, b}) P_{\xi}(b + \ii q_L)
\end{equation*}
where $S_{m, r} \to \PP^{m}$ is the unit sphere bundle in $(\olo_{\PP^m}(-1))^{\oplus r + 1} \to \PP^m$ with respect to the Hermitian metric inducing the Calabi form $\omega_{\xi_1, b}$, while $P(z)$ is the unique complex polynomial satisfying 
\begin{equation*}
P'_{\xi}(z) := \del_z P_{\xi}(z)= (\xi + z)^m z^r,\,P(0) = 0,
\end{equation*}
i.e. 
\begin{equation*}
P_{\xi}(z) = \int^{z}_0 (\xi + s)^m s^r ds.
\end{equation*}
Note that $P_{\xi}(z) = \frac{\xi^m z^{r + 1}}{r + 1} + O(z)$ as $z \to 0$.

In the case of our line bundles $L_i$, with corresponding cohomological parameters $b$, $q$, $q'$, as $b \to 0$ we also have $q, q' \to 0$ by construction, so there are expansions
\begin{equation*}
P_{\xi}(b + \ii q) = \frac{\xi^m (b + \ii q)^{r + 1}}{r + 1} + O(b),\,P_{\xi}(b + \ii q') = \frac{\xi^m (b + \ii q')^{r + 1}}{r + 1} + O(b)  
\end{equation*}
as $b \to 0$. 

Suppose that the phase $\psi$ of $\xi = \rho e^{\ii \psi}$ satisfies $m\psi + \frac{r+1}{2} \pi \in (-\pi, 0) \mod 2\pi$. Then, by the above expansions and elementary geometry in the lower half plane, the argument of $P_{\xi}(b + \ii q)$ is \emph{increasing} with respect to small variations of $\xi_1 = \Rea \xi$ for sufficiently small $b$, namely
\begin{equation*}
\del_{\xi_1} \arg P_{\xi}(b + \ii q) > 0 
\end{equation*}
for $|\xi_1|, b < \varepsilon$ for sufficiently small $\varepsilon > 0$. Similarly, since we are assuming that $r$ is even, we have $m\psi - \frac{r+1}{2} \pi  = m\psi + \frac{r+1}{2} - (r+1)\pi \in (0, \pi) \mod 2\pi$, and by elementary geometry in the upper half plane we find
\begin{equation*}
\del_{\xi_1} \arg P_{\xi}(b + \ii q') < 0. 
\end{equation*}

It follows that with these assumptions the inequality $\xi'_1 <  \xi_1$, depending on a small variation $\xi'_1$ of the K\"ahler parameter $\xi_1$, is equivalent to 
\begin{equation*}
\arg Z(L_1) < \arg Z(L_2[1]).
\end{equation*}

The same argument applies when $m\psi + \frac{r+1}{2} \pi \in (0, \pi) \mod 2\pi$, showing that $\xi'_1 <  \xi_1$ is equivalent to $\arg Z(L_2) < \arg Z(L_1[1])$ in this case. 

Finally we can translate these inequalities for central charges to inequalities for periods, as in the proof of Theorem \ref{MainThmIntro}.

Property $(iii)$ follows from the same argument given in the proof of Lemma \ref{PropertiesLem}.
\end{proof}
\bibliographystyle{abbrv}
 \bibliography{biblio_dHYM}
 
\noindent SISSA, via Bonomea 265, 34136 Trieste, Italy;\\
Institute for Geometry and Physics (IGAP), via Beirut 2, 34151 Trieste, Italy\\
jstoppa@sissa.it

\end{document}